\documentclass{lmcs} 


\usepackage{hyperref}
\usepackage{amssymb}
\usepackage{prftree}

\DeclareMathOperator{\Frm}{FORM}
\DeclareMathOperator{\Fun}{FUN}
\DeclareMathOperator{\Rel}{REL}
\DeclareMathOperator{\Var}{VAR}

\newcommand{\Class}{\mathrm{C}}

\renewcommand{\phi}{\varphi}

\newcommand{\Cut}{\textnormal{Cut}}
\newcommand{\Refl}{\textnormal{Refl}}

\renewcommand{\L}{\ensuremath{\mathbf{L}}}
\newcommand{\Lp}{\ensuremath{\mathbf{Lp}}}
\newcommand{\M}{\ensuremath{\mathbf{M}}}
\newcommand{\G}{\ensuremath{\mathbf{G}}}

\newcommand{\K}{\ensuremath{\mathbf{K}}}
\newcommand{\N}{\ensuremath{\mathbf{N}}}
\newcommand{\Gm}{\ensuremath{\mathbf{G3m}}}
\newcommand{\Gmp}{\ensuremath{\mathbf{G3mp}}}
\newcommand{\Gg}{\ensuremath{\mathbf{G3g}}}
\newcommand{\Gi}{\ensuremath{\mathbf{G3i}}}
\newcommand{\Gip}{\ensuremath{\mathbf{G3ip}}}
\newcommand{\Gc}{\ensuremath{\mathbf{G3c}}}
\newcommand{\Gcp}{\ensuremath{\mathbf{G3cp}}}
\newcommand{\Gcs}{\ensuremath{\mathbf{G3cs}}}
\newcommand{\Gcsp}{\ensuremath{\mathbf{G3csp}}}
\newcommand{\Gk}{\ensuremath{\mathbf{G3k}}}

\newcommand{\Maximum}{\ensuremath{\mathrm{M}}}

\usepackage{calligra}
\DeclareMathAlphabet{\mathcalligra}{T1}{calligra}{m}{n}
\DeclareFontShape{T1}{calligra}{m}{n}{<->s*[2.2]callig15}{}
\newcommand{\gleft}{\ell}
\newcommand{\gright}{\mkern-2mu\mathcalligra{r}\mkern2mu}

\begin{document}
		
		\title{Unifying Conservation as Translation for General Calculi}
		\titlecomment{{\lsuper*}This paper is an extended and substantially revised version of~\cite{Fellin26}.
			The main theorem of~\cite{Fellin26} corresponds to
			Theorem~\ref{thm-kolmogorov} here, but is no longer taken as the starting
			point: it is recovered as a particular case of the broader framework
			developed in the present paper. The generalisation to arbitrary calculi,
			the treatment of sequent translations via pairs
			$(\gleft,\gright)$, and the additional applications and conservation
			results are new to the present paper.}
		
		\author[G.~Fellin]{Giulio Fellin}	
		\address{Dipartimento di Informatica, Università degli Studi di Verona, Strada le Grazie 15, 37134, Verona, VR, Italy}	
		\email{giulio.fellin@univr.it}  
		\lmcsorcid{0000-0002-3179-7521}
		
		\begin{abstract}
			\noindent We present an abstract framework for conservation and translation theorems between logical calculi. Unlike previous approaches, our setting does not require the underlying consequence relations to satisfy structural properties such as cut, allowing in particular for cut-free calculi. Moreover, we study translations between arbitrary calculi rather than only from a stable extension to the original calculus. Translations are formulated at the level of sequents by pairs of functions acting on antecedents and succedents, and the framework is developed for multi-succedent calculi while subsuming the single-succedent case. This yields a uniform treatment of classical results including negative translations, minimality theorems, Orevkov's conservation classes, and a characterisation of the least logic satisfying Kuroda's double negation theorem.
		\end{abstract}
		
		\maketitle
		
		The relationship between classical and constructive logics has long been a central topic in proof theory and the philosophy of mathematics. Various syntactic translations and logical extensions have been developed to reconcile classical reasoning with constructive principles, leading to a rich body of \emph{conservation theorems} that clarify how much of classical logic can be recovered within constructive systems. These conservation results typically rely on the mechanism of \emph{double negation}.
		
		Negative translations have found applications in computer science~\cite{Griffin00}, set theory~\cite{Aczel01}, arithmetic and analysis~\cite{Spector62}, and have contributed to techniques such as program extraction~\cite{SchwichtenbergWainer12} and proof mining~\cite{Kohlenbach08}. More recently, a series of works~\cite{FellinSchusterWessel22,FellinSchuster21} extended Glivenko-type conservation theorems from concrete logical systems to arbitrary sets equipped with consequence relations. These generalisations rely on the observation that double negation is an instance of a nucleus and, more fundamentally, on the fact that nuclei interact well with the structural rules of reflexivity and transitivity (or, in proof-theoretic terms, the initial sequent and cut) when passing from an extension of minimal logic to classical logic~\cite{FellinNegriSchuster2026,BorsettoFellinUustaluWan2026}. As special cases, they recover classical results of Segerberg~\cite{Segerberg68} and Kuroda~\cite{Kuroda51}. In parallel,~\cite{FellinSchuster25,Fellin26} developed a corresponding abstract framework for negative translations, encompassing the classical translations of Kolmogorov~\cite{Kolmogorov25}, Gödel~\cite{Godel33}, Gentzen~\cite{Gentzen36}, and Kuroda~\cite{Kuroda51}, as well as Glivenko's theorem itself~\cite{Glivenko29}. The two lines of generalisation, however, remained separate.
		
		The approach presented here abstracts one step further. Rather than working with entailment relations satisfying structural assumptions, we consider arbitrary calculi, thereby encompassing in particular cut-free systems and calculi that need not satisfy cut or even reflexivity. Moreover, instead of restricting attention to translations from the stable extension of a calculus back to the original calculus, we study translations between arbitrary pairs of calculi. A further conceptual shift is that translations are no longer defined merely on formulae but on sequents, by means of pairs of functions $(\gleft,\gright)$ acting independently on the antecedent and succedent. Finally, our framework is formulated for multi-succedent calculi, while naturally subsuming the single-succedent setting considered in previous work. These generalisations considerably broaden the scope of the theory: they recover, in a uniform setting, minimality results such as Segerberg's claim and Kuroda's application of the double negation shift, encompass the negative translations of Kolmogorov, Gödel, Gentzen, and Kuroda, include Orevkov's theorems on conservation classes, and provide a characterisation of the least logic satisfying Kuroda's double negation theorem.
		
		\section{A brief account of conservation and translations}
		\label{sec-history}
		
		Classical logic adheres to the principles of the excluded middle and explosion. The former asserts that every proposition must be either true or false ($A \vee \neg A$), while the latter states that from a contradiction, any proposition can be derived ($\bot \to A$). In the early 20th century, L.E.J.~Brouwer~\cite{Brouwer07,Brouwer08} criticised the principle of the excluded middle for its nonconstructive nature: showing that the nonexistence of an object leads to a contradiction allows one to conclude that the object exists, without explicitly constructing it. Rejecting the principle of excluded middle, Brouwer introduced intuitionistic logic, later formalised by Andrey Kolmogorov~\cite{Kolmogorov25}, Valery Glivenko~\cite{Glivenko28,Glivenko29}, and Arend Heyting~\cite{Heyting30}. Intuitionistic logic satisfies two key properties: the \emph{existential property} (if there exists an object satisfying a certain property, then we can explicitly provide a witness) and the \emph{disjunction property} (if a disjunction $A \vee B$ is provable, then either $A$ is provable or $B$ is provable).
		
		\subsection{Glivenko's theorems}
		
		In 1928--1929, Valery Glivenko published two brief but influential notes~\cite{Glivenko28,Glivenko29}, written largely in response to criticisms directed at intuitionistic logic by Marcel Barzin and Alfred Errera \cite{BarzinErrera27}. Barzin and Errera had misunderstood Brouwer's position, suggesting that intuitionism implied the existence of a ``third'' truth value---something neither true nor false---for statements whose truth had not yet been established.%
		\footnote{For a more detailed history of the Barzin--Errera controversy, we refer to \cite{Thiel88}.}
		
		Glivenko set out to clarify this misconception and, in doing so, gave the first formal presentation of intuitionistic logic that included the principle of explosion. He proved that the principle of excluded middle, $A \vee \neg A$, is not refutable in intuitionistic logic; indeed, one can prove $\neg\neg(A \vee \neg A)$. This shows that intuitionism does not deny the excluded middle outright, but merely refuses to assert it in the absence of a constructive proof. In this sense, the law of excluded middle behaves like the parallel postulate in geometry: it is independent of the other axioms, neither provable nor disprovable without additional assumptions.
		Thus, Glivenko's analysis refuted the idea that intuitionistic logic admits a ``third'' logical value.
		
		In the course of this work, Glivenko established two results that later became known as \emph{Glivenko's theorem}. The first states that a propositional formula is provable in classical logic if and only if its double negation is provable in intuitionistic logic. The second states that a negated propositional formula is intuitionistically provable if and only if it is classically provable.
		
		Glivenko’s theorems marks the starting point of a broader line of investigation into the relationship between classical and constructive logics. In its wake, a number of conservation results were established that aim to generalise and refine its central insight. These developments fall, broadly speaking, into three main strands.
		
		First, \emph{negative translation theorems}, such as those due to Kolmogorov, G\"odel, and Gentzen, provide systematic embeddings of classical logic into intuitionistic or minimal logic by means of formula translations. Second, \emph{minimality results}, exemplified by the work of Kuroda and Segerberg, identify the weakest logical systems in which Glivenko-type correspondences 
		continue to hold. Third, results on \emph{conservation classes} characterise classes of formulae for which classical derivability already implies intuitionistic derivability without any translation.
		
		Taken together, these results show that classical principles can often be faithfully interpreted—or even partially recovered—within constructive frameworks, provided one either restricts the class of formulae considered or modifies them in a controlled way. In this sense, they collectively extend Glivenko’s original insight into a rich and systematic theory of the interaction between classical and constructive reasoning.%
		\footnote{Strictly speaking, Kolmogorov’s result predates Glivenko’s theorem. Glivenko was aware of Kolmogorov’s work—he mentions it in a letter to Heyting prior to his publications \cite{Troelstra90}—but does not cite it in his final note \cite{Glivenko29}. It is nevertheless unlikely that Glivenko’s theorem was entirely uninfluenced by Kolmogorov’s earlier ideas.}
		
		\subsection{Negative translations}
		
		In 1925, Andrey Kolmogorov published a seminal paper \cite{Kolmogorov25} aimed at clarifying the logical foundations of intuitionism. Kolmogorov accepted the intuitionist critique of classical logic, especially the rejection of the principle of the excluded middle. However, he went further by also rejecting the principle of explosion, arguing that it lacks constructive meaning. His reasoning was that deriving an arbitrary conclusion from a contradiction does not correspond to constructing a proof of that conclusion, and hence explosion is epistemically and constructively illegitimate. Kolmogorov described this more restricted logical system as intuitionistic logic, but today it is known as \emph{minimal logic}---a term introduced by Ingebrigt Johansson in 1937 \cite{Johansson37}, likely independently of Kolmogorov's earlier work.
		
		Kolmogorov then introduced a key syntactic technique, now known as the \emph{Kolmogorov negative translation}, which provides a faithful embedding of classical logic into minimal logic. The translation $K$ recursively maps each formula $A$ to a formula $K(A)$ by double negating atomic formulae and systematically placing double negations around each logical connective and quantifier. Using modern notation, it is defined by
		\[
		\begin{array}{rclcrcl}
			K(P) &:=& \neg\neg P & \text{if $P$ is atomic}, 
			\\[1mm]
			K(A * B) &:=& \neg\neg (K(A) * K(B)) &\text{if } * \in \{\wedge, \vee, \to\}, 
			\\[1mm]
			K(\mathsf{Q} x\, A) &:=& \neg\neg \mathsf{Q} x\, K(A) &\text{if } \mathsf{Q} \in \{\forall, \exists\}.
		\end{array}
		\]
		Kolmogorov's negative translation theorem states that $A$ and $K(A)$ are classically equivalent and that $A$ is provable in classical logic if and only if $K(A)$ is provable in minimal logic. 
		
		Kolmogorov viewed minimal logic as the true logic of mathematical construction, while classical principles like excluded middle and explosion correspond to what he called \emph{pseudotruths}: a formula $A$ is a pseudotruth if its translation $K(A)$ is provable in minimal logic. These pseudotruths behave like truths within classical reasoning but do not reflect genuine constructive knowledge. According to Kolmogorov, the reason classical reasoning avoids contradictions is that it operates on pseudotruths rather than actual constructive proofs. This radical yet elegant perspective anticipated much later developments in proof theory, realisability, and constructive mathematics.
		
		In 1933, Kurt G\"odel and Gerhard Gentzen independently introduced negative translations embedding classical arithmetic into intuitionistic arithmetic, apparently unaware of Kolmogorov's earlier work. Although Gentzen completed his manuscript in that same year, he withdrew it at the galley proof stage upon learning of G\"odel's result~\cite{Godel33}. The work was eventually published three years later, in 1936~\cite{Gentzen36}. Gentzen's translation~$G$ inserts double negations in front of atomic formulae, disjunctions, and existential quantifiers, and just propagates through the other constructions, i.e.,
		\[
		\begin{array}{rclcrcl}
			G(P) &:=& \neg\neg P & \text{if $P$ is atomic}, 
			\\[1mm]
			G(A * B) &:=& G(A) * G(B) &\text{if } * \in \{\wedge, \to\}, 
			\\[1mm]
			G(A \to B) &:=& \neg\neg (G(A) \vee G(B)),
			\\[1mm]
			G(\forall x\, A) &:=& \forall x\, G(A),
			\\[1mm]
			G(\exists x\, A) &:=& \neg\neg \exists x\, G(A).
		\end{array}
		\]
		His negative translation theorem states that a formula~$A$ is classically provable if and only if $G(A)$ is provable in intuitionistic logic, and that $A$ and $G(A)$ are classically equivalent. G\"odel's variant $G'$ placed double negation also in front of implications:
		\[
		\begin{array}{rclcrcl}
			G'(A \to B) &:=& \neg\neg (G'(A) \to G'(B)), 
		\end{array}
		\]
		thus occupying a middle ground between Kolmogorov's and Gentzen's schemes. His result asserts that~$A$ is provable classically if and only if $G(A)$ is provable in intuitionistic logic. Although G\"odel's and Gentzen's translations are distinct in their original form, they are often conflated under the term \emph{G\"odel--Gentzen negative translation}~\cite{TroelstraSchwichtenberg00,FerreiraOliva13}, typically referring to Gentzen's version.
		It is also well known that, over minimal logic, the translations $K$, $G$, and $G'$ are pairwise equivalent~\cite{FerreiraOliva13}: for every formula $A$, we have
		\(
		\vdash\; K(A)\;\leftrightarrow\; G(A)\;\leftrightarrow\; G'(A).
		\)
		
		Moreover, an important contribution of both G\"odel and Gentzen is that their negative translations extend beyond logic to full arithmetic. In fact, already Kolmogorov had recognised the possibility of applying this approach to mathematical theories, not just logical formulae, thereby laying the groundwork for later developments in constructive interpretations of classical mathematics.
		
		In 1951, Sigekatu Kuroda published a summary \cite{Kuroda51} of results from his earlier Japanese papers of 1947--1948 \cite{kuroda1947sagaku,kuroda1948aristotle}. In these works, he introduced a new negative translation and used it to extend Glivenko's theorem from propositional to first-order predicate logic. Kuroda observed that Glivenko's original double negation theorem remains valid in predicate logic only for formulae without universal quantifiers. More generally, however, a corresponding conservation result can be obtained by applying what is now known as the \emph{Kuroda negative translation}. This translation inserts double negations not only in front of the entire formula but also beneath every occurrence of the universal quantifier.
		More precisely, Kuroda's negative translation is given by $\neg\neg J$, where $J$ is defined recursively by
		\[
		\begin{array}{rclcrcl}
			J(P) &:=& P & \text{if $P$ is atomic}, 
			\\[1mm]
			J(A * B) &:=& J(A) * J(B) &\text{if } * \in \{\wedge, \vee, \to\}, 
			\\[1mm]
			J(\exists x\, A) &:=& \exists x\, J(A), 
			\\[1mm]
			J(\forall x\, A) &:=& \forall x\, \neg\neg J(A).
		\end{array}
		\]
		Like Kolmogorov's translation, Kuroda's translation preserves classical equivalence: $A$ and $\neg\neg J(A)$ are classically equivalent. Moreover, $A$ is provable in classical logic if and only if $\neg\neg J(A)$ is provable in intuitionistic logic. 
		A distinctive feature of Kuroda's translation is its factorisation into the composition
		$A \mapsto J(A) \mapsto \neg\neg J(A)$.
		This separation of the translation into an inner transformation $J$ followed by an outer double negation plays an important role in later formulations of Kuroda's theorem.
		
		Unlike Kolmogorov's translation, Kuroda's original translation does not apply to minimal logic. A modified version suitable for minimal logic was introduced by Chetan Murthy in 1990~\cite{Murthy90} and is sometimes referred to as the \emph{minimal Kuroda negative translation}~\cite{Berg19,FerreiraOliva13}.
		For a more detailed account of negative translations, we refer to~\cite{FerreiraOliva13}.
		
		\subsection{Minimality results}
		\label{sec-Kuroda}
		
		In his aforementioned works (1947--1951), Kuroda further isolated a key principle needed to lift Glivenko's result to the predicate level: the 
		\emph{double negation shift for universal quantification},
		\(
		\forall x\,\neg\neg A(x)\ \to\ \neg\neg\forall x\,A(x),
		\)
		a principle that intuitionistic predicate logic does not validate.
		Assuming this shift, he showed that classical provability of a formula entails the intuitionistic provability of its \emph{ordinary double negation}---that is, of the same $\neg\neg$-translation used in Glivenko's original theorem and not Kuroda's own $\neg\neg J(A)$.
		In Kuroda's terminology, a formula is then ``free of contradiction'' exactly when its double negation is intuitionistically derivable, echoing Kolmogorov's notion of pseudotruth.
		
		In 1968, Krister Segerberg~\cite{Segerberg68} introduced what is now called \emph{Glivenko's logic}, originally denoted $JP'$. He defined it as minimal propositional logic extended by the double negation of the principle of explosion, $\neg\neg(\bot \to A)$. Segerberg also asserted%
		\footnote{Segerberg did not provide a proof or citation for this claim. A rigorous proof was given only decades later by Odintsov~\cite{Odintsov02}.}
		that $JP'$ is the weakest logic $\L$ satisfying Glivenko's theorem: a formula $A$ is provable in classical logic if and only if $\neg\neg A$ is provable in $\L$. This establishes Glivenko's logic as a foundational system with distinctive structural properties.
		
		Glivenko's logic offers a novel treatment of falsity and negation, presenting a compromise between intuitionistic logic and minimal logic. It rejects the principle of explosion, aligning with the constructive nature of minimal logic, while simultaneously refusing to outright negate the principle whenever falsity is unprovable. This ensures weak but crucial consistency: not everything becomes provable. As such, Glivenko's logic strikes a balance that is both philosophically compelling and potentially foundational for alternative reasoning frameworks.
		The characterisation of Glivenko's logic as the weakest logic satisfying Glivenko's theorem underscores its remarkable strength. By enabling classical reasoning to operate under specific constraints, particularly in proofs of negated formulae, it provides a refined framework for integrating classical and constructive methods. This duality enhances its philosophical and technical relevance.
		
		After its introduction in 1968, this topic logic received limited attention. Beyond Segerberg's foundational paper, only two brief notes in the early 1970s \cite{Woodruff70,Goldblatt74} touched upon the topic, after which it vanished from the literature for nearly three decades, until it was revived by Odintsov, mainly for technical purposes \cite{Odintsov02}.
		
		The study of Glivenko's logic becomes significantly more intricate when extended to predicate logic. Combining insights from Kuroda and Segerberg shows that, to obtain Glivenko's predicate logic---the minimal extension of minimal predicate logic for which Glivenko's theorem holds---one must incorporate both the double negation of explosion and the double negation shift. As a result, this logic is incomparable with intuitionistic predicate logic: it is neither weaker nor stronger, but rather a distinct intermediate logic lying strictly between minimal and classical predicate logics \cite{BartoliFellinTesi2026}.
		
		\subsection{Conservation classes}
		
		A third line of work concerns conservation classes, that is, classes of sequents for which classical derivability entails intuitionistic (or minimal) derivability. An early and fundamental example is given by the aforementioned variant of Glivenko's theorem \cite{Glivenko29}, which states that a negated propositional formula is intuitionistically provable if and only if it is classically provable.
		
		In 1968, Vladimir Orevkov \cite{Orevkov68} established a series of conservativity theorems for classical logic over intuitionistic and minimal first-order logics with equality. In particular, he identified seven distinct classes of single-succedent sequents—the so-called \emph{Glivenko sequent classes}—which are characterised syntactically by the absence of formulae of certain \emph{polarities}.
		
		Given a sequent $\Gamma \Rightarrow \Delta$, we stipulate that all formulae in $\Delta$ are \emph{positive} (or \emph{have a positive polarity}) and all formulae in $\Gamma$ are \emph{negative} (or \emph{have a negative polarity}). The polarity of subformulae is defined inductively by:
		\begin{itemize}
			\item If $A * B$ is positive (resp.\ negative), where $*\in\{\wedge,\vee\}$, then both $A$ and $B$ are positive (resp.\ negative);
			\item If $\mathsf{Q} x\, A$ is positive (resp.\ negative), where $\mathsf{Q}\in\{\forall,\exists\}$, then $A$ is positive (resp.\ negative);
			\item If $A \to B$ is positive (resp.\ negative), then $A$ is negative (resp.\ positive) and $B$ is positive (resp.\ negative).
		\end{itemize}
		
		The seven classes are:
		\begin{enumerate}
			\item no positive occurrences of $\to$ or $\forall$.
			\item no positive occurrences of $\to$; no negative occurrences of $\vee$.
			\item no positive occurrences of $\to$; no negative occurrences of $\forall$.
			\item no positive occurrences of $\vee$ or $\exists$; no negative occurrences of $\to$.
			\item no positive occurrences of $\to$ (except as $\neg$), $\forall$, or $\vee$;
			no negative occurrences of $\to$.
			\item no positive occurrences of $\to$ (except as $\neg$) or $\vee$;
			no negative occurrences of $\to$ or $\vee$.
			\item no positive occurrences of $\to$ (except as $\neg$) or $\vee$;
			no negative occurrences of $\to$ or $\forall$.
		\end{enumerate}
		
		Orevkov proved that, for each of these classes, classical derivability entails intuitionistic derivability. Moreover, the classes are optimal in the sense that any collection of sequents enjoying this conservativity property is contained in one of the seven classes.
		
		In recent years, streamlined proofs of conservativity for several of Orevkov’s Glivenko classes have been obtained. For example, Sara Negri~\cite{Negri03} gave a notably concise and purely logical proof for the first Glivenko class in the setting of coherent theories, using $\mathbf{G3}$-style sequent calculi. This approach was subsequently extended to geometric theories~\cite{Negri16}, and later generalised to encompass all seven Glivenko sequent classes~\cite{Negri21,FellinNegriOrlandelli22}.
		
		\section{Calculi}
		
		Let $S$ be a set and let $S^*$ denote the set of finite lists over $S$.
		For $\Gamma,\Delta\in S^*$, we write
		$\Gamma \Rightarrow \Delta$
		for the pair $(\Gamma,\Delta)$, called a \emph{sequent}.
		Sequents occurring in rules may contain metavariables ranging over elements of $S$ as well as context variables ranging over finite lists in $S^*$. Thus a sequent may consist of a fixed part together with a variable part that is instantiated uniformly when the rule is applied.
		
		A \emph{rule} is a rule schema of the form
		\[
		\prftree[r]{$r$}
		{\{\Gamma_i\Rightarrow\Delta_i : i\in I\}}
		{\Gamma\Rightarrow\Delta},
		\]
		where $I\subseteq\mathbb N$. An \emph{instance} of a rule is obtained by simultaneously substituting the metavariables occurring in the rule.
		
		A \emph{calculus} $\L$ on $S$ is a set of rules. An \emph{$\L$-derivation} is a finite well-founded proof tree whose nodes are sequents and whose every inference is an instance of a rule in $\L$. We write
		\[
		\L\vdash \Gamma\Rightarrow\Delta
		\]
		if there exists an $\L$-derivation with conclusion $\Gamma\Rightarrow\Delta$.
		
		A rule $r$ is \emph{derivable} in $\L$ if, for every instance of $r$, there exists an $\L$-derivation of its conclusion from its premisss. Following \cite[p.~21]{NegriPlato11}, a rule is \emph{admissible} in $\L$ if derivability is closed under the rule, that is, whenever all premisss of an instance of the rule are derivable in $\L$, so is its conclusion.
		
		Among the rules that can appear in a calculus, we want to mention the single-conclusion rules of reflexivity and cut:
		\begin{align*}
			\prftree[r]{\Refl}{A\Rightarrow A}
			&&
			\prftree[r]{\Cut}{\Gamma\Rightarrow A}{\Delta_1,A,\Delta_2\Rightarrow B}{\Delta_1,\Gamma,\Delta_2\Rightarrow B}
		\end{align*}
		These have been shown to work well in the context of nuclei \cite{BorsettoFellinUustaluWan2026,FellinNegriSchuster2026} and for this reason in previous work have always been taken as primitive. However, in the present article we do not require it.
		
		\begin{rem}
			\label{rmk-calculiorders}
			Let $\L,\M$ be calculi on $S$. We write
			\[
			\begin{aligned}
				\L\sqsubseteq\M
				&\quad\text{if every rule of $\L$ is derivable in $\M$,}
				\\[0.3em]
				\L\preceq\M
				&\quad\text{if every rule of $\L$ is admissible in $\M$,}
				\\[0.3em]
				\L\leq\M
				&\quad\text{if }\L\vdash\Gamma\Rightarrow\Delta
				\text{ implies }\M\vdash\Gamma\Rightarrow\Delta.
			\end{aligned}
			\]	
			
			We also write $\L\cong\M$ if both $\L\preceq\M$ and $\M\preceq\L$.
			
			Then it is immediate to see that the relations $\sqsubseteq$, $\preceq$, and $\leq$ are all transitive relations, and moreover each of the following items entails the next one:
			\[
			\L\subseteq\M
			\quad\Longrightarrow\quad
			\L\sqsubseteq\M
			\quad\Longrightarrow\quad
			\L\preceq\M
			\quad\Longrightarrow\quad
			\L\leq\M.
			\]
			
			We notice that none of the implications can be reversed \cite{FellinNegriSchuster2026}.
			To show this, consider $S=\{\top,\bot\}$ and the rules
			\begin{align*}
				\prftree[r]{Refl$^+$}{\Gamma_1,A,\Gamma_2\rhd A}
				&&
				\prftree[r]{$r_1$}{\top,\bot\rhd \bot}
				&&
				\prftree[r]{$r_2$}{{}\rhd \top}{{}\rhd \bot}
				&&
				\prftree[r]{$r_3$}{\Gamma \rhd \top}
			\end{align*}
			and define $\L_0:=\{\textnormal{Refl}^+\}$, $\L_1:=\{\textnormal{Refl}^+,r_1\}$, $\L_2:=\{\textnormal{Refl}^+,r_2\}$, and $\L_3:=\{\textnormal{Refl}^+,r_3\}$.
			First, observe that the rule $r_2$ is admissible in $\L_0$ since its premiss cannot be derived so any conclusion whatsoever follows. This means that $\L_2\preceq\L_0\subseteq\L_3$, which, as observed above, gives $\L_2\leq\L_3$.
			On the other hand, $r_2$ is not admissible in $\L_3$: the premiss is always derivable by $r_3$, but the conclusion can never be derived.
			Thus $\L_2\npreceq\L_3$.%
			\footnote{
				The example provided here is a concrete example showing a rule that is admissible in a calculus but does not remain admissible in every inductive extension. However, it is quite simplified and artificial.
				A more typical example is a sequent calculus in which the cut rule is admissible; upon extending the calculus with additional axioms, however, cut often ceases to be admissible \cite{NegriPlato01,NegriPlato11}.
			}
			We have observed that $\L_2\preceq\L_0$. On the other hand, we have a rule, that is $r_2$, which is derivable in $\L_2$ but not in $\L_0$. Thus $\L_2\not\sqsubseteq\L_0$.
			Lastly, we observe that $r_1$ is derivable in $\L_0$ by $\textnormal{Refl}^+$. Therefore $\L_1\sqsubseteq\L_0$. However, it is clear that $\L_1\nsubseteq\L_0$.
		\end{rem}
		
		\begin{rem}
			For uniformity, our framework is presented in terms of multi-succedent calculi. This entails no loss of generality, since any single-succedent calculus can be viewed as a multi-succedent calculus in which no inference rule produces a sequent with multiple formulae in the succedent. Consequently, every derivation remains single-succedent, and all the results developed below apply equally to the single-succedent setting.
		\end{rem}
		
		\subsection{Conservation as Translation}
		
		We observe that the three families of conservation theorems discussed in Section~\ref{sec-history} share a common abstract pattern:
		\[
		\M \vdash \Gamma \Rightarrow A,\text{ if and only if }\L \vdash \gleft(\Gamma) \Rightarrow \gright(A),
		\]
		where $\gleft$ and $\gright$ are operations on formulae, extended to contexts. In the cases of interest, $\M$ is typically a calculus for classical logic, while the choice of $\L$, $\gleft$, and $\gright$ varies according to the type of result under consideration:
		\begin{enumerate}
			\item Minimality results in the style of Kuroda and Segerberg take $\L$ to be an extension of a chosen base calculus (intuitionistic or minimal), with $\gleft$ the identity and $\gright(A):=\neg\neg A$;
			
			\item Translation results in the style of Kolmogorov, Gödel, and Gentzen take $\L$ to be a calculus for minimal (or, in some cases, intuitionistic) logic, with $\gleft=\gright=k$, where $k$ is a suitable negative translation;
			
			\item Results concerning conservation classes such as Orevkov's take $\L$ to be a calculus for intuitionistic logic and restrict attention to those sequents $\Gamma\Rightarrow A$ satisfying $\gleft(\Gamma)=\Gamma$ and $\gright(A)=A$.
		\end{enumerate}
		Our goal is to identify the structural features common to these situations and thereby obtain a uniform generalisation within an abstract framework.
		
		So let $\L$ be a calculus on $T$, and let $\M$ be a calculus on $S$. We consider sequents of the form $\gleft(\Gamma) \Rightarrow \gright(\Delta)$, where $\gleft, \gright \colon S^* \to T^*$. In this setting, the pair $(\gleft,\gright)$ induces a mapping from sequents derivable in $\M$ to sequents derivable in $\L$. We therefore regard $(\gleft,\gright)$ as a morphism and write $(\gleft,\gright) \colon \M \to \L$.
		
		\begin{defi}
			Let $\gleft,\gright\colon S^*\to T^*$.
			Given a rule $r$ on $S$, we define $r_{(\gleft,\gright)}$ on $T$ as follows:
			\begin{align*}
				\prftree[r]{\textnormal{$r$}}{\{\Gamma_i\Rightarrow \Delta_i\colon i \in I\}}{\Gamma\Rightarrow \Delta}
				&&\rightsquigarrow&&
				\prftree[r]{\textnormal{$r_{(\gleft,\gright)}$}}{\{\gleft(\Gamma_i)\Rightarrow \gright(\Delta_i)\colon i \in I\}}{\gleft(\Gamma)\Rightarrow \gright(\Delta)}
			\end{align*}
			We also write 
			$\M_{(\gleft,\gright)}:=\{r_{(\gleft,\gright)}\colon r\in\M \}.$
		\end{defi}
		
		\begin{defi}\
			Let $\L$ be a calculus on $T$, let $\M$ be a calculus on $S$, and let $(\gleft,\gright)\colon\M\to\L$.
			\begin{enumerate}
				\item The map $(\gleft,\gright)$ is a \emph{sound translation} from $\M$ to $\L$ if
				for all sequents $\Gamma \Rightarrow \Delta$,
				\[
				\M \vdash \Gamma \Rightarrow \Delta
				\;\text{ implies }\;
				\L \vdash \gleft(\Gamma) \Rightarrow \gright(\Delta).
				\]
				\item The map $(\gleft,\gright)$ is a \emph{faithful translation} from $\M$ to $\L$ if
				for all sequents $\Gamma \Rightarrow \Delta$,
				\[
				\L \vdash \gleft(\Gamma) \Rightarrow \gright(\Delta)
				\;\text{ implies }\;
				\M \vdash \Gamma \Rightarrow \Delta.
				\]
				\item A function $k\colon S^*\to T^*$ is a (\emph{sound} / \emph{faithful}) \emph{translation} from $\M$ to $\L$ if the pair $(k,k)$ is a (sound / faithful) translation from $\M$ to $\L$.
			\end{enumerate}
		\end{defi}
		
		\begin{prop}
			\label{thm-main-t}
			Let $\L$ be a calculus on $T$, let $\M$ be a calculus on $S$, and let $(\gleft,\gright)\colon \M\to\L$.
			Define $\G := \L \cup \M_{(\gleft,\gright)}$.
			Then:
			\begin{enumerate}
				\item \textbf{Soundness.}
				\label{item-main-t-sound}
				The map $(\gleft,\gright)\colon \M \to \G$ is a sound translation.
				
				Moreover, $(\gleft,\gright)\colon \M \to \K$ is sound for any $\K$ such that $\G\preceq\K$.
				
				\item 
				\textbf{Minimality.}
				\label{item-main-t-min}
				Given $\K$ such that $\L\preceq\K$
				and $(\gleft,\gright)\colon\N\to\K$ is a sound and faithful translation, we have that
				$\G\preceq\K$.
			\end{enumerate}
		\end{prop}
		
		\begin{proof}\
			\begin{enumerate}
				\item We prove soundness from $\M$ to $\G$ by induction on the derivation of $\M\vdash\Gamma\Rightarrow \Delta$. Suppose that it has been derived by an application of $r\in\M$. Then
				\begin{align*}
					\prftree[r]{\textnormal{$r$}}{\cdots}{\M\vdash\Gamma_i\Rightarrow \Delta_i}{\cdots}{\M\vdash\Gamma\Rightarrow \Delta}
					&&\rightsquigarrow&&
					\prftree[r]{\textnormal{$r_{(\gleft,\gright)}$}}{\cdots}{
						\prftree[r]{i.h.}{\M\vdash\Gamma_i\Rightarrow \Delta_i}{\G\vdash\gleft(\Gamma_i)\Rightarrow \gright(\Delta_i)}
					}{\prfassumption\cdots}{\G\vdash\gleft(\Gamma)\Rightarrow \gright(\Delta)}
				\end{align*}
				where $r_{(\gleft,\gright)}$ can be applied since $\M_{(\gleft,\gright)}\subseteq\G$.
				
				Then soundness from $\M$ to $\K$ follows by transitivity of $\leq$ since $\M\leq\G_{(\gleft,\gright)}\leq\K_{(\gleft,\gright)}$.
				
				\item Let $r\in\M$. Then it is admissible in $\N$ and
				\begin{equation*}
					\prftree[r]{\textnormal{Sound}}{
						\prftree[r]{\textnormal{$r$}}{\cdots}{
							\prftree[r]{\textnormal{Faithful}}
							{\K\vdash\gleft(\Gamma_i)\Rightarrow \gright(\Delta_i)}
							{\N\vdash\Gamma_i\Rightarrow \Delta_i}
						}{\prfassumption\cdots}
						{\N\vdash\Gamma\Rightarrow\Delta}
					}{\K\vdash\gleft(\Gamma)\Rightarrow \gright(\Delta)}
				\end{equation*}
				which means that $r_{(\gleft,\gright)}$ is admissible in $\K$. Therefore $\M_{(\gleft,\gright)}\preceq\K$ and, since $\L\preceq\K$, we conclude $\G\preceq\K$.
				\qedhere
			\end{enumerate}
		\end{proof}
		
		We observe that in all three families of conservation theorems discussed in Section~\ref{sec-history}, the language of calculi $\L$ and $\M$ is always the same, which in our setting means $S=T$.
		In this situation, we find a very simple sufficient condition for faithfulness:
		
		\begin{lem}
			\label{lem-Mfaithful}
			Let $\L$ and $\M$ be calculi on $S$ such that $\L\preceq\M$, and
			let $(\gleft,\gright)\colon \M\to\M$ be a sound and faithful translation.
			Define $\G:=\L\cup\M_{(\gleft,\gright)}$.
			Then:
			\begin{enumerate}
				\item $\L\preceq\G\preceq\M$.
				\item \textbf{Faithfulness.}
				The map $(\gleft,\gright)\colon \M\to\G$ is a faithful translation.
			\end{enumerate}
		\end{lem}
		
		\begin{proof}
			First, notice that $\G\preceq\M$ by Proposition \ref{thm-main-t}(\ref{item-main-t-min}) with $\M$ in place of $\K$.
			Now let $\G\vdash\gleft(\Gamma)\Rightarrow\gright(\Delta)$. Since $\G\preceq\M$, we get $\M\vdash\gleft(\Gamma)\Rightarrow\gright(\Delta)$ which, by faithfulness from $\M$ to $\M$, gives $\M\vdash\Gamma\Rightarrow\Delta$. Therefore we have faithfulness from $\M$ to $\G$.
		\end{proof}
		
		\begin{thm}
			\label{thm-main}
			Let $\L$ and $\M$ be calculi on $S$, and let $(\gleft,\gright)\colon \M\to\L$ such that $(\gleft,\gright)\colon \M\to\M$ is sound and faithful.
			Define $\G := \L \cup \M_{(\gleft,\gright)}$.
			Then:
			\begin{enumerate}
				\item \textbf{Soundness.}
				\label{item-main-sound}
				The map $(\gleft,\gright)\colon \M \to \G$ is a sound translation.
				
				Moreover, $(\gleft,\gright)\colon \M \to \K$ is sound for any $\K$ such that $\G\preceq\K$.
				
				\item \textbf{Faithfulness.}
				\label{item-main-faithful}
				The map $(\gleft,\gright)\colon \M \to \K$ is a faithful translation for any $\K\preceq\M$.
				
				In particular, this holds for $\K=\L$ and for $\K=\G$.
				
				\item 
				\textbf{Minimality.}
				\label{item-main-min}
				Given $\K$ such that $\L\preceq\K$
				and $(\gleft,\gright)\colon\M\to\K$ is a sound and faithful translation, we have that $\G\preceq\K$.
			\end{enumerate}
		\end{thm}
		
		\begin{proof}
			Soundness and minimality are particular cases of their counterparts in Proposition \ref{thm-main-t}, while faithfulness comes from Lemma \ref{lem-Mfaithful}.
		\end{proof}
		
		\begin{rem}[Conservation as translation]
			The first item of Theorem \ref{thm-main} gives a proof-theoretic characterisation of conservation: $\M$ is conservative over $\K$ modulo the translation $(\gleft,\gright)$ precisely when the rules of $\M$, translated through $(\gleft,\gright)$, are admissible in $\K$.
		\end{rem}
		
		We don't have a similar criterion for faithfulness in the general case in which not necessarily $S=T$. However, the following characterisation can be useful.
		
		\begin{rem}
			Assume that at least one sequent is derivable in $\M$ and that the map $(\gleft,\gright)\colon\M\to\L$ is a sound translation.
			Then $(\gleft,\gright)$ is faithful if and only if it is \emph{Lindenbaum-injective}, i.e. if
			for all sequents $\Gamma \Rightarrow \Delta$ and $\Gamma' \Rightarrow \Delta'$,
			\[
			\begin{aligned}
				\prftree[doubleline]{\L\vdash\gleft(\Gamma)\Rightarrow\gright(\Delta)}{\L\vdash\gleft(\Gamma')\Rightarrow\gright(\Delta')}
				&&
				\text{ implies }
				&&
				\prftree[doubleline]{\M\vdash\Gamma\Rightarrow\Delta}{\M\vdash\Gamma'\Rightarrow\Delta'}
			\end{aligned}
			\]
			where the double line denotes the rule together with its inverse.
			In fact, assume that $(\gleft,\gright)\colon \M \to \L$ is Lindenbaum-injective, and suppose that $\L\vdash\gleft(\Gamma)\Rightarrow\gright(\Delta)$. Consider a sequent $\Gamma'\Rightarrow\Delta'$ such that $\M\vdash\Gamma'\Rightarrow\Delta'$ (it exists by hypothesis). Then, by soundness, $\L\vdash\gleft(\Gamma')\Rightarrow\gright(\Delta')$.
			Since the two translated sequents are both derivable in $\L$, and $\M\vdash\Gamma'\Rightarrow\Delta'$, then Lindenbaum injectivity lets us conclude $\M\vdash\Gamma\Rightarrow\Delta$.
			Conversely, assume that $(\gleft,\gright)\colon \M \to \L$ is faithful. Suppose that $\L\vdash\gleft(\Gamma)\Rightarrow\gright(\Delta)$ if and only if $\L\vdash\gleft(\Gamma')\Rightarrow\gright(\Delta')$
			and assume $\M\vdash\Gamma\Rightarrow\Delta$. By soundness, we get $\L\vdash\gleft(\Gamma)\Rightarrow\gright(\Delta)$, which by assumption gives $\L\vdash\gleft(\Gamma')\Rightarrow\gright(\Delta')$. By faithfulness, we conclude $\M\vdash\Gamma'\Rightarrow\Delta'$.
		\end{rem}
		
		\subsection{Logical calculi}
		
		\begin{figure}
			
			\begin{description}
				\item[\(\Gm\)] 
				\setlength{\abovedisplayskip}{.5ex}
				\setlength{\belowdisplayskip}{.5ex}
				
				\[
				\begin{array}{c@{\qquad}c@{\qquad}c}
					\multicolumn{3}{c}{\textbf{Logical rules}}\\[.5ex]
					&
					\multicolumn{2}{c}{\prftree[r]{R$\top$}{\Gamma\Rightarrow\top}}
					\\[1ex]
					
					\prftree[r]{L$\vee$}{A,\Gamma\Rightarrow C}
					{B,\Gamma\Rightarrow C}
					{A\vee B,\Gamma\Rightarrow C}
					&
					\prftree[r]{R$\vee_1$}{\Gamma\Rightarrow A}
					{\Gamma\Rightarrow A\vee B}
					&
					\prftree[r]{R$\vee_2$}{\Gamma\Rightarrow B}
					{\Gamma\Rightarrow A\vee B}
					\\[1ex]
					\prftree[r]{L$\wedge$}{A,B,\Gamma\Rightarrow C}
					{A\wedge B,\Gamma\Rightarrow C}
					&
					\multicolumn{2}{c}{\prftree[r]{R$\wedge$}{\Gamma\Rightarrow A}
						{\Gamma\Rightarrow B}
						{\Gamma\Rightarrow A\wedge B}}
					\\[1ex]
					
					\prftree[r]{L$\to$}{\Gamma\Rightarrow A}
					{B,\Gamma\Rightarrow C}
					{A\to B,\Gamma\Rightarrow C}
					&
					\multicolumn{2}{c}{\prftree[r]{R$\to$}{A,\Gamma\Rightarrow B}
						{\Gamma\Rightarrow A\to B}}
					\\[1ex]
					
					\prftree[r]{L$\exists$}{A[y/x],\Gamma\Rightarrow C}
					{\exists x\,A,\Gamma\Rightarrow C}
					&
					\multicolumn{2}{c}{\prftree[r]{R$\exists$}{\Gamma\Rightarrow A[t/x]}
						{\Gamma\Rightarrow\exists x\,A}}
					\\[-.2ex]
					
					\multicolumn{1}{c}{\text{\scriptsize($y$ fresh)}}
					&
					&
					\\[1ex]
					
					\prftree[r]{L$\forall$}{A[t/x],\forall x\,A,\Gamma\Rightarrow C}
					{\forall x\,A,\Gamma\Rightarrow C}
					&
					\multicolumn{2}{c}{\prftree[r]{R$\forall$}{\Gamma\Rightarrow A[y/x]}
						{\Gamma\Rightarrow\forall x\,A}}
					\\[-.2ex]
					
					&
					\multicolumn{2}{c}{\text{\scriptsize($y$ fresh)}}
					\\[1ex]
					
					\multicolumn{3}{c}{\textbf{Structural rules}}\\[.5ex]
					
					\prftree[r]{IS}{\Gamma_1,P,\Gamma_2\Rightarrow P}
					&
					\multicolumn{2}{c}{\prftree[r]{LE}{\Gamma_1,A,B,\Gamma_2\Rightarrow C}
						{\Gamma_1,B,A,\Gamma_2\Rightarrow C}}
				\end{array}
				\]
				
				\rule{.85\textwidth}{.4pt}
				
				\item[\(\Gi\)] \(\Gm\), plus 
				\[
				\prftree[r]{L$\bot$}{\bot,\Gamma\Rightarrow C}
				\]
				
				\rule{.85\textwidth}{.4pt}
				
				\item[\(\Gg\)] \(\Gm\), plus Cut and
				\[
				\begin{array}{c@{\qquad\qquad}c}
					\prftree[r]{R$\to_{(1,\neg\neg)}$}
					{A,\Gamma\Rightarrow\neg\neg B}
					{\Gamma\Rightarrow\neg\neg(A\to B)}
					&
					\prftree[r]{R$\forall$}
					{\Gamma\Rightarrow\neg\neg A[y/x]}
					{\Gamma\Rightarrow\neg\neg\forall x\,A}
					\\[-.2ex]
					&
					\text{\scriptsize($y$ fresh)}
				\end{array}
				\]
				
				\rule{.85\textwidth}{.4pt}
				
				\item[\(\Gcs\)] \(\Gm\), plus Cut and
				\[
				\prftree[r]{L$\neg\neg$}
				{A,\Gamma\Rightarrow C}
				{\neg\neg A,\Gamma\Rightarrow C}
				\]
				
				\rule{.85\textwidth}{.4pt}
			\end{description}
			
			\caption{The calculi \(\Gm,\Gi,\Gg,\Gcs\).}
			\label{fig:calculi}
		\end{figure}
		\begin{figure}
			
			\begin{description}
				\item[\(\Gc\)]
				\[
				\begin{array}{c@{\qquad\qquad}c}
					\multicolumn{2}{c}{\textbf{Logical rules}}\\[.5ex]
					
					\prftree[r]{L$\bot$}{\bot,\Gamma\Rightarrow\Delta}
					&
					\prftree[r]{R$\top$}{\Gamma\Rightarrow\Delta,\top}
					\\[1ex]
					
					\prftree[r]{L$\vee$}
					{A,\Gamma\Rightarrow\Delta}
					{B,\Gamma\Rightarrow\Delta}
					{A\vee B,\Gamma\Rightarrow\Delta}
					&
					\prftree[r]{R$\vee$}
					{\Gamma\Rightarrow\Delta,A,B}
					{\Gamma\Rightarrow\Delta,A\vee B}
					\\[1ex]
					
					\prftree[r]{L$\wedge$}
					{A,B,\Gamma\Rightarrow\Delta}
					{A\wedge B,\Gamma\Rightarrow\Delta}
					&
					\prftree[r]{R$\wedge$}
					{\Gamma\Rightarrow\Delta,A}
					{\Gamma\Rightarrow\Delta,B}
					{\Gamma\Rightarrow\Delta,A\wedge B}
					\\[1ex]
					
					\prftree[r]{L$\to$}
					{\Gamma\Rightarrow\Delta,A}
					{B,\Gamma\Rightarrow\Delta}
					{A\to B,\Gamma\Rightarrow\Delta}
					&
					\prftree[r]{R$\to$}
					{A,\Gamma\Rightarrow\Delta,B}
					{\Gamma\Rightarrow\Delta,A\to B}
					\\[1ex]
					
					\prftree[r]{L$\exists$}
					{A[y/x],\Gamma\Rightarrow\Delta}
					{\exists x\,A,\Gamma\Rightarrow\Delta}
					&
					\prftree[r]{R$\exists$}
					{\Gamma\Rightarrow\Delta,\exists x\,A,A[t/x]}
					{\Gamma\Rightarrow\Delta,\exists x\,A}
					\\[-.2ex]
					
					\multicolumn{1}{c}{\text{\scriptsize($y$ fresh)}}
					&
					\\[1ex]
					
					\prftree[r]{L$\forall$}
					{A[t/x],\forall x\,A,\Gamma\Rightarrow\Delta}
					{\forall x\,A,\Gamma\Rightarrow\Delta}
					&
					\prftree[r]{R$\forall$}
					{\Gamma\Rightarrow\Delta,A[y/x]}
					{\Gamma\Rightarrow\Delta,\forall x\,A}
					\\[-.2ex]
					
					&
					\multicolumn{1}{c}{\text{\scriptsize($y$ fresh)}}
					\\[1ex]
					
					\multicolumn{2}{c}{\textbf{Structural rules}}\\[.5ex]
					
					\prftree[r]{IS}
					{\Gamma_1,P,\Gamma_2\Rightarrow\Delta_1,P,\Delta_2}
					&
					\\[1ex]
					
					\prftree[r]{LE}
					{\Gamma_1,A,B,\Gamma_2\Rightarrow\Delta}
					{\Gamma_1,B,A,\Gamma_2\Rightarrow\Delta}
					&
					\prftree[r]{RE}
					{\Gamma\Rightarrow\Delta_1,A,B,\Delta_2}
					{\Gamma\Rightarrow\Delta_1,B,A,\Delta_2}
				\end{array}
				\]
				
				\rule{.85\textwidth}{0.4pt}
			\end{description}
			
			\caption{The calculus \(\Gc\).}
			\label{fig-Clas}
		\end{figure}
		
		Our main examples are the calculi for minimal, intuitionistic, Glivenko,
		and classical predicate logic, all formulated over the set $\Frm$ of
		first-order formulae.
		
		The sets of terms and formulae are generated inductively by
		\[
		\begin{aligned}
			t &::= x
			\mid f^n(t_1,\ldots,t_n),\\
			A &::= P^n(t_1,\ldots,t_n)
			\mid t_1=t_2
			\mid \bot
			\mid \top
			\mid A\wedge A
			\mid A\vee A
			\mid A\to A
			\mid \forall x\,A
			\mid \exists x\,A,
		\end{aligned}
		\]
		where $f^n\in \Fun^\mathcal{S}_n$, $P^n\in \Rel^\mathcal{S}_n$, and
		$x,x_1,\ldots,x_n\in \Var$.
		As usual, we use $x,y,z$ for variables, $t,s,r$ for terms, $P,Q,R$ for atomic formulae, $A,B,C,D$ for arbitrary formulae, and $\Gamma,\Delta,\Phi$ for contexts.
		We adopt the standard conventions concerning parentheses.
		
		Negation is introduced as the abbreviation
		$\neg A := A\to\bot.$
		
		The notions of \emph{free} and \emph{bound occurrences} of variables are
		standard. Given a formula $A$, we write $A[t/x]$ for the result of
		substituting $t$ for each free occurrence of $x$ in $A$, provided that $t$
		is free for $x$ in $A$, i.e., provided that no occurrence of $t$ becomes
		bound by a quantifier as a result of the substitution.
		
		Figure~\ref{fig:calculi} presents the calculus $\Gm$ for minimal predicate
		logic and its extensions $\Gi$, $\Gg$, and $\Gcs$ for intuitionistic,
		Glivenko, and classical predicate logic, respectively. Each of the latter
		calculi is obtained from $\Gm$ by adding suitable rules.
		
		We also consider the multi-succedent calculus $\Gc$ for classical predicate logic is shown separately in Figure~\ref{fig-Clas}.
		It is well known that \(\Gcs\) is the \emph{trace} of \(\Gc\), that is,
		\[
		\Gc \vdash \Gamma \Rightarrow A
		\quad\text{if and only if}\quad
		\Gcs \vdash \Gamma \Rightarrow A.
		\]
		In other words, \(\Gcs\) and \(\Gc\) are equivalent with respect to single-conclusion derivability. We shall therefore freely switch between them whenever one presentation is technically more convenient than the other.
		
		The calculi considered here are defined using initial sequents involving atomic
		formulae $P$, for structural reasons. Nevertheless, in both the single-succedent
		and the multi-succedent cases, the corresponding more general forms are
		admissible, respectively,
		\[
		\prftree[r]{Refl$^+$}{\Gamma_1,A,\Gamma_2\Rightarrow A}
		\qquad
		\prftree[r]{Refl$^+$}
		{\Gamma_1,A,\Gamma_2\Rightarrow\Delta_1,A,\Delta_2}
		\]
		
		For any predicate calculus $\L$, we denote by $\Lp$ its propositional fragment, obtained by removing the quantifiers together with all inference rules governing them.
		
		The calculi considered here are adaptations of those in~\cite{NegriPlato01}. Because antecedents and succedents are represented as lists of formulae, exchange must be included explicitly among the structural rules. Finally, while \Gm, \Gi, and {\Gc} enjoy cut elimination and are therefore formulated without a cut rule, {\Gg} and {\Gcs} do not; accordingly, cut is included among their primitive rules.
		
		\begin{rem}
			\label{rmk-multisingle}
			In many cases, conservation results are formulated only for single-succedent sequents, that is, sequents of the form $\Gamma \Rightarrow A$. Such sequents fit naturally within the intuitionistic single-succedent calculi $\Gm$, $\Gg$, and $\Gi$. By contrast, when working in classical logic, it is preferable to use the multi-succedent calculus $\Gc$, which enjoys a number of desirable proof-theoretic properties, most notably cut elimination.
			
			We therefore require a way of passing from a multi-succedent setting to a single-succedent one. The key observation is that, in classical logic, the connectives $\bot$ and $\vee$ allow us to encode succedents containing either zero or several formulae. More precisely:
			\begin{itemize}
				\item $\Gc\vdash\Gamma\Rightarrow {}$ if and only if $\Gc\vdash\Gamma\Rightarrow \bot$.
				\item $\Gc\vdash\Gamma\Rightarrow D_1,\ldots,D_n$ if and only if $\Gc\vdash\Gamma\Rightarrow D_1\vee\cdots\vee D_n$.
			\end{itemize}
			
			To implement this observation, whenever we pass from a multi-succedent calculus to a single-succedent one, we define the function $\gright$ so that
			\[
			\begin{array}{rclcrcl}
				\gright(\epsilon) &:=& \bot,
				\\[1mm]
				\gright(C_1,\ldots,C_n) &:=& \phi(C_1 \vee \cdots \vee C_n)
			\end{array}
			\]
			where $\epsilon$ is the empty list, for some $\phi\colon \Frm\to \Frm$.
			Since $\gright$ acts on the right-hand side of a sequent, these clauses allow any sequent with multiple succedent formulae to be represented by a single-succedent sequent that is classically equivalent to the original one.
			
			On the other hand, if we do not need to pass from a multi- to a single-succedent calculus (e.g. if we translate from $\Gcs$ to $\Gm$), then we can just impose that $\gright$ acts pointwise on the items of the list, viz.
			\[
			\begin{array}{rclcrcl}
				\gright(\epsilon) &:=& \epsilon,
				\\[1mm]
				\gright(C_1,\ldots,C_n) &:=& \gright(C_1), \ldots, \gright(C_n).
			\end{array}
			\]
		\end{rem}
		
		\section{Nuclei and minimality results}
		
		In this section and the ones that follow, we show how the conservation theorems surveyed in Section~\ref{sec-history} arise uniformly as applications of our main result, Theorem~\ref{thm-main}.
		
		We begin with Glivenko's theorem and related minimality results, such as those of Kuroda and Segerberg. To this end, we introduce the notion of a \emph{nucleus} and the associated translations. An application of Theorem~\ref{thm-main} then yields a uniform treatment of these results, recovering and further generalising the conservation theorems established in \cite{FellinSchuster21,FellinSchuster25}.
		
		\begin{defi}\
			Let $\L$ be a calculus on $T$, let $\M$ be a calculus on $S$, and let $(\gleft,\gright)\colon\M\to\L$.
			\begin{enumerate}
				\item The map $(\gleft,\gright)$ is \emph{expansive} in $\L$ if the following are admissible in $\L$:
				\begin{align*}
					\prftree[r]{R$(\gleft,\gright)$}{\Gamma\Rightarrow  \Delta_1, \gleft(\Phi),\Delta_2}{\Gamma\Rightarrow \Delta_1, \gright(\Phi), \Delta_2}
					&&\text{and}&&
					\prftree[r]{L$(\gright,\gleft)$}{\Gamma_1, \gright(\Phi), \Gamma_2\Rightarrow \Delta}{\Gamma_1,\gleft(\Phi), \Gamma_2 \Rightarrow \Delta}
				\end{align*}
				\item The map $(\gleft,\gright)$ is \emph{stable} in $\L$ if $(\gright,\gleft)$ is expansive in $\L$. 
				\item The map $(\gleft,\gright)$ is \emph{progressively monotone} in $\L$ if the following is admissible in $\L$:
				\begin{align*}
					\prftree[r]{Lp$(\gleft,\gright)$}
					{\gleft(\Gamma_1),\gleft(\Phi),\gleft(\Gamma_2)\rhd \gright(\Delta)}
					{\gleft(\Gamma_1),\gright(\Phi),\gleft(\Gamma_2)\rhd \gright(\Delta)}
				\end{align*}
				\item The map $(\gleft,\gright)$ is \emph{nuclear} in $\L$ if it is expansive and progressively monotone in $\L$.
				\item A function $j\colon T^*\to T^*$ is (\emph{expansive} / \emph{stable} / \emph{progressively monotone} / a \emph{nucleus}) in $\L$ if the pair $(1,j)$ is (expansive / stable / progressively monotone / nuclear) in $\L$, where $1\colon T^*\to T^*$ is the identity map.
				\item The \emph{$j$-stable extension} of $\L$ is defined as
				$\Maximum_j(\L):=\L\cup\{\Refl,\Cut,\textnormal{L$(1,j)$},\textnormal{R${(1,j)}$}\}$.
			\end{enumerate}
		\end{defi}
		
		\begin{rem}
			\label{rmk-jstable}
			If $\gleft,\gright\colon S^*\to S^*$ are both expansive and stable in $\M$, then
			$(\gleft,\gright)\colon \M\to\M$ is a faithful translation.
		\end{rem}
		
		\begin{thm}
			\footnote{
				This theorem is a version of~\cite[Theorem~3.5]{FellinSchuster25}.
				The original formulation appears in~\cite[Corollary~3.11]{FellinSchuster21}, where it was derived as a corollary of Theorem~3.8 (corresponding here to Corollary~\ref{cor-glivenko}).
			}
			\label{thm-kuroda}
			Let $\L$ be a calculus on $S$ such that $\Cut$ and $\Refl$ are admissible in $\L$,
			let $j\colon S^*\to S^*$, and define
			$
			\G := \L \cup \L_{(1,j)} \cup \{\Refl,\Cut,\textnormal{Lp$(1,j)$},\textnormal{R${(1,j)}$}\}.
			$
			Then:
			\begin{enumerate}
				\item \textbf{Soundness \& Faithfulness.}
				The map $(1,j)\colon \Maximum_j(\L)\to\G$ is a sound and faithful translation, i.e.
				\[
				\Maximum_j(\L)\vdash\Gamma\Rightarrow\Delta
				\quad\text{if and only if}\quad
				\G\vdash \Gamma\Rightarrow j(\Delta).
				\]
				
				\item \textbf{Minimality.}
				Given $\K$ such that $\L\preceq\K$ and $(1,j)\colon\Maximum_j(\L)\to\K$ is a sound and faithful translation, we have $\G\preceq\K$.
			\end{enumerate}
		\end{thm}
		
		\begin{proof}
			First, let us recall that, having \Refl\ and \Cut, Lp$(1,j)$ and R$(1,j)$ are equiadmissible with $\Cut_{(1,j)}$ and $\Refl_{(1,j)}$, respectively \cite{BorsettoFellinUustaluWan2026,FellinNegriSchuster2026}.
			Moreover, L$(1,j)_{(1,j)}$ and R$(1,j)_{(1,j)}$ are instances of Lp$(1,j)$ and R$(1,j)$, respectively.
			Therefore $\G\cong\L\cup(\Maximum_j(\L))_{(1,j)}$. Given the observations in Remark \ref{rmk-jstable}, the claim follows from Theorem \ref{thm-main} with $(\gleft,\gright)=(1,j)$.
		\end{proof}
		
		The following addition to Theorem \ref{thm-kuroda} is also worth noting:
		
		\begin{rem}[Maximality]
			Given $\K$ such that $\G\preceq\K$ and $(1,j)\colon\Maximum_j(\L)\to\K$ is a faithful translation, we have $\K\leq\Maximum_j(\L)$. In fact, R$(1,j)$ directly gives us $\K\leq\K_{(1,j)}$, and faithfulness can be written as $\K_{(1,j)}\leq\Maximum_j(\L)$; we conclude by transitivity of $\leq$.
		\end{rem}
		
		\begin{cor}
			\footnote{This corollary is corresponds to~\cite[Theorem~3.8]{FellinSchuster21}.}
			\label{cor-glivenko}
			Let $\L$ be a calculus on $S$ such that $\Cut$ and $\Refl$ are admissible in $\L$.
			Let $j$ such that $\L_{(1,j)}\cup\{\textnormal{Lp$(1,j)$},\textnormal{R${(1,j)}$}\} \preceq \L$.
			Then
			$\Maximum_j(\L)\vdash\Gamma\Rightarrow\Delta$ if and only if $\L\vdash \Gamma\Rightarrow j(\Delta)$.
		\end{cor}
		
		\begin{proof}
			Direct consequence of Theorem \ref{thm-kuroda}.
		\end{proof}
		
		\subsection{Glivenko's theorem}
		
		The principal—and motivating—example of a nucleus in $\Gm$ is \emph{double negation} $\neg\neg$, also referred to as the \emph{Glivenko nucleus} or the \emph{Glivenko translation}~\cite{FerreiraOliva13,Berg19,FellinSchuster21,FellinSchuster25}. 
		In this treatment, as explained in Remark \ref{rmk-multisingle}, we need to extend the definition of $\neg\neg$ to lists as 
		\[
		\begin{aligned}
			\neg\neg(\epsilon) &:= \epsilon,
			\\
			\neg\neg(C_1,\ldots,C_n) &:= \neg\neg C_1 , \ldots, \neg\neg C_n.
		\end{aligned}
		\]
		It is well known and straightforward to verify that $\neg\neg$ forms a nucleus (see e.g.~\cite{FellinSchuster25}), and we therefore omit the proof. 
		Accordingly, in what follows we shall make use of the rules Lp$(1,\neg\neg)$ and R$(1,\neg\neg)$.
		It is immediate to check that $\Maximum_{\neg\neg}(\Gm)\cong\Gcs$ and $\Maximum_{\neg\neg}(\Gmp)\cong\Gcsp$.
		
		\begin{lem}
			\label{lem-glivenko}
			$(\Gip)_{(1,\neg\neg)}\cup\{\textnormal{Lp$(1,\neg\neg)$},\textnormal{R${(1,\neg\neg)}$}\} \preceq \Gip$.
		\end{lem}
		
		\begin{proof}
			First, recall that $\Gip$ admits \Cut\ \cite{NegriPlato01}.
			
			\noindent
			\textbf{The rules \textnormal{Lp$(1,\neg\neg)$} and \textnormal{R$(1,\neg\neg)$}.}	
			Direct consequences
			of the fact that $\neg\neg$ is a nucleus.
			
			\noindent
			\textbf{The rule \textnormal{IS$_{(1,\neg\neg)}$}.}	
			Follows directly from IS and R$(1,\neg\neg)$.
			
			\noindent
			\textbf{The rules \textnormal{LE$_{(1,\neg\neg)}$}, \textnormal{L$\wedge_{(1,\neg\neg)}$}, and \textnormal{L$\vee_{(1,\neg\neg)}$}.}	
			They are admissible in $\Gip$, since they are respectively instances of
			LE, L$\wedge$, and L$\vee$.
			
			\noindent
			\textbf{The rule \textnormal{R$\wedge_{(1,\neg\neg)}$}.}	
			Let us consider R$\wedge_{(1,\neg\neg)}$:
			\[
			\prftree[r]{R$\wedge_{(1,\neg\neg)}$}
			{\Gamma \Rightarrow \neg\neg A}{}
			{\Gamma \Rightarrow \neg\neg B}
			{\Gamma \Rightarrow \neg\neg(A \wedge B)}
			\]
			We have
			\[
			\prftree[r]{Cut}{
				\prftree[r]{R$\wedge$}{
					\prfassumption{\Gamma \Rightarrow \neg\neg A}
				}{
					\prfassumption{\Gamma \Rightarrow \neg\neg B}
				}{\Gamma \Rightarrow \neg\neg A \wedge \neg\neg B}
			}{
				\prftree[r]{L$\wedge$}{
					\prftree[r]{Lp$(1,\neg\neg)$}{
						\prftree[r]{R$(1,\neg\neg)$}{
							\prftree[r]{R$\wedge$}{
								\prftree[r]{Refl$^+$}{A ,  B \Rightarrow A}
							}{
								\prftree[r]{Refl$^+$}{A ,  B \Rightarrow B}
							}{A ,  B \Rightarrow A \wedge B}
						}{A ,  B \Rightarrow \neg\neg(A \wedge B)}
					}{\neg\neg A , \neg\neg B \Rightarrow \neg\neg(A \wedge B)}
				}{\neg\neg A \wedge \neg\neg B \Rightarrow \neg\neg(A \wedge B)}
			}{\Gamma \Rightarrow \neg\neg(A \wedge B)}
			\]
			
			\noindent
			\textbf{The rules \textnormal{R$\top_{(1,\neg\neg)}$}, \textnormal{R$\vee_{1(1,\neg\neg)}$}, and \textnormal{R$\vee_{2(1,\neg\neg)}$}.}	
			Can be handled analogously.
			
			\noindent
			\textbf{The rule \textnormal{L$\to_{(1,\neg\neg)}$}.}	
			Let us consider L$\to_{(1,\neg\neg)}$:
			\[
			\prftree[r]{L$\to_{(1,\neg\neg)}$}
			{\Gamma \Rightarrow \neg\neg A}
			{B, \Gamma \Rightarrow \neg\neg C}
			{A\to B, \Gamma \Rightarrow \neg\neg C}
			\]
			We have
			\[
			\prftree[r]{Cut}{
				\prftree[r]{R$\to$}{
					\prftree[r]{Lp$(1,\neg\neg)$}{
						\prftree[r]{R$(1,\neg\neg)$}{
							\prftree[r]{LE}{
								\prftree[r]{L$\to$}{
									\prftree[r]{Refl$^+$}{A \Rightarrow A}
								}{
									\prftree[r]{Refl$^+$}{ B , A \Rightarrow B}
								}{A\to B, A \Rightarrow B}
							}{A, A\to B \Rightarrow B}
						}{A, A\to B \Rightarrow \neg\neg B}
					}{\neg\neg A, A\to B \Rightarrow \neg\neg B}
				}{A\to B \Rightarrow \neg\neg A \to \neg\neg B}
			}{
				\prftree[r]{L$\to$}{
					\prfassumption{\Gamma \Rightarrow \neg\neg A}
				}{
					\prftree[r]{Lp$(1,\neg\neg)$}{B, \Gamma \Rightarrow \neg\neg C}{\neg\neg B, \Gamma \Rightarrow \neg\neg C}
				}{\neg\neg A \to \neg\neg B, \Gamma \Rightarrow \neg\neg C}
			}{A\to B, \Gamma \Rightarrow \neg\neg C}
			\]
			
			\noindent
			\textbf{The rule \textnormal{L$\bot_{(1,\neg\neg)}$}.}	
			It is admissible in $\Gip$,
			since it is an instance of L$\bot$.
			
			\noindent
			\textbf{The rule \textnormal{R$\to_{(1,\neg\neg)}$}.}
			First, we establish the two auxiliary sequents
			\[
			\neg\neg B , \neg(A\to B) \Rightarrow \bot
			\qquad\text{and}\qquad
			\neg(A\to B) \Rightarrow \neg\neg A.
			\]
			For the first one, we have
			\[
			\prftree[r]{\textnormal{L$\to$}}{
				\prftree[r]{R$\to$}{
					\prftree[r]{LE}{
						\prftree[r]{\textnormal{L$\to$}}{
							\prftree[r]{\textnormal{R$\to$}}{
								\prftree[r]{Refl$^+$}{B, A \Rightarrow B}
							}{ B \Rightarrow A\to B}
						}{
							\prftree[r]{Refl$^+$}{\bot, B \Rightarrow \bot}
						}{\neg(A\to B), B \Rightarrow \bot}
					}{B, \neg(A\to B) \Rightarrow \bot}
				}{\neg(A\to B) \Rightarrow \neg B}
			}{
				\prftree[r]{Refl$^+$}{\bot, \neg(A\to B) \Rightarrow \bot}
			}{\neg\neg B , \neg(A\to B) \Rightarrow \bot}
			\]
			and
			\[
			\prftree[r]{\textnormal{R$\to$}}{
				\prftree[r]{LE}{
					\prftree[r]{L$\to$}{
						\prftree[r]{R$\to$}{
							\prftree[r]{LE}{
								\prftree[r]{L$\to$}{
									\prftree[r]{Refl$^+$}{A \Rightarrow A}
								}{
									\prftree[r]{L$\bot$}{\bot , A \Rightarrow B}
								}{\neg A, A \Rightarrow B}
							}{A, \neg A \Rightarrow B}
						}{\neg A \Rightarrow A \to B}
					}{
						\prftree[r]{Refl$^+$}{\bot , \neg A \Rightarrow \bot}
					}{\neg(A\to B), \neg A \Rightarrow \bot}
				}{\neg A, \neg(A\to B) \Rightarrow \bot}
			}{\neg(A\to B) \Rightarrow \neg\neg A}
			\]
			
			These two derivations allow us to derive
			$\neg\neg(A\to B)$ from
			$\neg\neg A\to\neg\neg B$.
			Indeed, we have
			\[
			\prftree[r]{Cut}{
				\prftree[r]{R$\to$}{
					\prftree[r]{Lp$(1,\neg\neg)$}{A, \Gamma \Rightarrow \neg\neg B}{\neg\neg A, \Gamma \Rightarrow \neg\neg B}
				}{\Gamma \Rightarrow \neg\neg A \to \neg\neg B}
			}{
				\prftree[r]{R$\to$}{
					\prftree[r]{LE}{
						\prftree[r]{\textnormal{L$\to$}}{
							\prftree[r]{}{\vdots}{\neg(A\to B) \Rightarrow \neg\neg A}
						}{
							\prftree[r]{}{\vdots}{\neg\neg B , \neg(A\to B) \Rightarrow \bot}
						}{\neg\neg A \to \neg\neg B , \neg(A\to B) \Rightarrow \bot}
					}{\neg(A\to B), \neg\neg A \to \neg\neg B \Rightarrow \bot}
				}{\neg\neg A \to \neg\neg B \Rightarrow \neg\neg(A\to B)}
			}{\Gamma \Rightarrow \neg\neg(A\to B)}
			\qedhere
			\]
		\end{proof}
		
		\begin{prop}[Glivenko \cite{Glivenko29}]
			\label{prop-glivenko-1}
			$\Gcsp \vdash \Gamma \Rightarrow A$ if and only if $\Gip \vdash \Gamma \Rightarrow \neg\neg A$.
		\end{prop}
		\begin{proof}
			Direct application of Corollary~\ref{cor-glivenko} given Lemma~\ref{lem-glivenko}.
		\end{proof}
		
		\subsection{The Kuroda--Segerberg theorem}
		
		We now prove a version that simultaneously encompasses both Kuroda's theorem \cite{kuroda1947sagaku,kuroda1948aristotle,Kuroda51} and Segerberg's claim \cite{Segerberg68} (see also Subsection \ref{sec-Kuroda}), starting from minimal predicate logic.
		
		\begin{lem}\
			\label{lem-kuroda-segerberg}
			$\Gm_{(1,\neg\neg)}\cup\{\textnormal{Lp$(1,\neg\neg)$},\textnormal{R${(1,\neg\neg)}$}\} \preceq \Gg$.
		\end{lem}
		
		\begin{proof}
			We observe that the rules Lp$(1,\neg\neg)$, R$(1,\neg\neg)$, IS$_{(1,\neg\neg)}$, LE$_{(1,\neg\neg)}$, L$\wedge_{(1,\neg\neg)}$, L$\vee_{(1,\neg\neg)}$, R$\wedge_{(1,\neg\neg)}$, R$\top_{(1,\neg\neg)}$, R$\vee_{1(1,\neg\neg)}$, R$\vee_{2(1,\neg\neg)}$, L$\to_{(1,\neg\neg)}$ are handled as in Lemma \ref{lem-glivenko}.
			The rules L$\exists_{(1,\neg\neg)}$ and L$\forall_{(1,\neg\neg)}$ are instances of L$\exists$ and L$\forall$, respectively.
			The rule R$\exists_{(1,\neg\neg)}$ can be handled similarly as R$\wedge_{(1,\neg\neg)}$.
			Since $\Gg$ extends $\Gm$ with rules R$\to_{(1,\neg\neg)}$, R$\forall_{(1,\neg\neg)}$ and \Cut, we conclude that $\Gm_{(1,\neg\neg)}\cup\{\textnormal{Lp$(1,\neg\neg)$},\textnormal{R${(1,\neg\neg)}$}\} \preceq \Gg$.
		\end{proof}
		
		\begin{prop}
			\label{prop-Kuroda-Segerberg}
			Let $\K$ be a calculus such that
			$\Gm\preceq\K\preceq\Gcs$.
			Then:
			\begin{enumerate}
				\item If $\K\vdash \Gamma \Rightarrow \neg\neg A$, then $\Gcs\vdash \Gamma \Rightarrow A$.
				\item The following are equivalent:
				\begin{enumerate}
					\item If $\Gcs\vdash \Gamma \Rightarrow A$, then $\K\vdash \Gamma \Rightarrow \neg\neg A$.
					\item $\Gg\preceq\K$.
				\end{enumerate}
			\end{enumerate}
		\end{prop}
		\begin{proof}
			Direct application of Theorem~\ref{thm-kuroda} given Lemma~\ref{lem-kuroda-segerberg}.
		\end{proof}
		
		\section{Negative translations}
		
		We now turn to a generalised setting encompassing negative translations of the kind introduced by Kolmogorov, Gödel, Gentzen, and Kuroda, among others. By specialising our framework to the case in which
		$k := \gleft = \gright$
		and imposing a suitable collection of conditions on $k$, we obtain a particular case of Theorem~\ref{thm-main}. 
		
		\begin{defi}
			Let $\L$ and $\M$ be calculi on $S$ such that $\L\preceq\M$, and
			let $j,k\colon S^*\to S^*$.
			We say that $k$ is a \emph{proto-$j$-translation} from \M\ to \L\ if
			the function $k$ is a sound and faithful translation from $\M$ to $\M$ and
			the map $(k,kj)$ is both expansive and stable in \L.
		\end{defi}
		
		\begin{thm}
			\footnote{
				This result extends the scope of the
				main theorem of~\cite{Fellin26}, providing a more general and widely
				applicable account of generalised negative translations. It also
				generalises~\cite[Theorem 3.13]{FellinSchuster25}.
			}
			\label{thm-kolmogorov}
			Let $\L$ be a calculus on $S$ such that \Refl\ and \Cut\ are admissible in \L.
			Let $j,k\colon S^*\to S^*$ such that $k$ is a proto-$j$-translation from $\Maximum_j(\L)$ to \L, and $\L_{(k,k)} \preceq \L$.
			Then $k$ is a sound and faithful translation from $\Maximum_j(\L)$ to $\L$.
		\end{thm}
		
		\begin{proof}
			Notice that $\Refl_{(k,k)}$ and $\Cut_{(k,k)}$ are instances of \Refl\ and \Cut, respectively, and that L$(1,j)_{(k,k)}$ and R$(1,j)_{(k,k)}$ are direct consequences of $(k,kj)$ being stable and expansive in \L, respectively.
			Therefore $(\Maximum_j(\L))_{(k,k)}\preceq\L$, and the claim follows from Theorem \ref{thm-main}.
		\end{proof}
		
		\subsection{Glivenko's theorem, rephrased}
		
		Similarly to what done in \cite{Fellin26}, Corollary \ref{cor-glivenko} may also be obtained from Theorem \ref{thm-kolmogorov} by setting $k = j$ and carrying out further appropriate observations. In fact, we can derive the following variant of Glivenko's theorem (Proposition \ref{prop-glivenko-1}).
		
		\begin{prop}[Glivenko \cite{Glivenko29}]
			\label{prop-glivenko-2}
			The function $\neg\neg$ is a sound and faithful translation from $\Gcsp$ to $\Gip$.
		\end{prop}
		
		\begin{proof}
			Clearly, $\neg\neg$ is stable and expansive in~$\Gcsp \cong \Maximum_{\neg\neg}(\Gip)$. We also observe that $\Gip\vdash\neg\neg\neg\neg A\Rightarrow\neg\neg A$ and $\Gip\vdash\neg\neg A\Rightarrow\neg\neg\neg\neg A$ are direct consequences of the fact that $\neg\neg$ is a nucleus. Therefore $\neg\neg$ is a proto-$\neg\neg$-translation from $\Gcsp$ to $\Gip$. Moreover, by Lemma \ref{lem-glivenko} and the rules of a nucleus we get $\Gip_{(\neg\neg,\neg\neg)}\preceq\Gip$. We conclude by Theorem~\ref{thm-kolmogorov}.
		\end{proof}
		
		\subsection{Kolmogorov's negative translation theorem}
		
		Let us consider maps from $\Gcs$ to~$\Gm$.
		We inductively define the function $K$, corresponding to Kolmogorov's negative translation:
		\[
		\begin{array}{rclcrcl}
			K(P) &:=& \neg\neg P, 
			\\[1mm]
			K(A * B) &:=& \neg\neg (K(A) * K(B)) &\text{if } * \in \{\wedge, \vee, \to\}, 
			\\[1mm]
			K(\mathsf{Q} x\, A) &:=& \neg\neg \mathsf{Q} x\, K(A) &\text{if } \mathsf{Q} \in \{\forall, \exists\},
			\\[1mm]
			K(\epsilon) &:=& \epsilon,
			\\[1mm]
			K(C_1,\ldots,C_n) &:=& K(C_1), \ldots, K(C_n).
		\end{array}
		\]
		We observe that, for every formula $A$, and every term $t$,
		\(
		K(A[t/x])=(K(A))[t/x].
		\)
		In particular, the eigenvariable conditions of the quantifier rules are
		preserved under the transformation $K$.
		
		\begin{lem}
			\label{rmk-Kolmogorov-Gentzen}
			$K$ is a proto-$\neg\neg$-translation from $\Gcs$ to~$\Gm$.
		\end{lem}
		
		\begin{proof}
			It can be shown that $K$ is stable and expansive in~$\Gcs \cong \Maximum_{\neg\neg}(\Gm)$ by an easy induction argument, using the fact that $\neg\neg$ is stable and expansive in this extension and can therefore be “freely added and removed.”
			
			We now need to show that $\Gm\vdash K(A) \Rightarrow K(\neg\neg A)$.
			We observe that $K(A) = \neg\neg A'$ for some~$A'$. So the claim reduces to $\Gm \vdash \neg\neg A' \Rightarrow \neg\neg \neg\neg A'$, which is an easy consequence of the fact that $\neg\neg$ is a nucleus.
			Finally, we need to show that $\Gm\vdash K(\neg\neg A) \Rightarrow K(A)$.
			Again, $K(A) = \neg\neg A'$ for some~$A'$. The claim then reduces to $\Gm \vdash \neg\neg \neg\neg A' \Rightarrow \neg\neg A'$, which is an easy consequence of the fact that $\neg\neg$ is a nucleus.
		\end{proof}
		
		\begin{lem}
			\label{lem-Kolmogorov-Gentzen}
			$\Gm_{(K,K)}\preceq\Gm$.
		\end{lem}
		
		\begin{proof}\
			Again, recall that $\Gm$ admits \Cut\ \cite{NegriPlato01}.
			
			\noindent
			\textbf{The rule \textnormal{IS$_{(K,K)}$}.}
			This is an instance of Refl$^+$.
			
			\noindent
			\textbf{The rule \textnormal{R$\wedge_{(K,K)}$}.}
			Let us first consider the rule
			\[
			\prftree[r]{R$\wedge_{(K,K)}$}
			{K(\Gamma) \Rightarrow K(A)}{}
			{K(\Gamma) \Rightarrow K(B)}
			{K(\Gamma) \Rightarrow K(A \wedge B)}
			\]
			We obtain:
			\[\resizebox{\linewidth}{!}{\(
				\prftree[r]{\Cut}{
					\prftree[r]{R$\wedge$}{
						\prfassumption{K(\Gamma) \Rightarrow K(A)}
					}{
						\prfassumption{K(\Gamma) \Rightarrow K(B)}
					}{K(\Gamma) \Rightarrow K(A) \wedge K(B)}
				}{
					\prftree[r]{L$\wedge$}{
						\prftree[r]{dfn}{
							\prftree[r]{R$(1,\neg\neg)$}{
								\prftree[r]{R$\wedge$}{
									\prftree[r]{\textnormal{Refl$^+$}}{K(A), K(B) \Rightarrow K(A)}
								}{
									\prftree[r]{\textnormal{Refl$^+$}}{K(A), K(B) \Rightarrow K(B)}
								}{K(A), K(B) \Rightarrow K(A) \wedge K(B)}
							}{K(A), K(B) \Rightarrow \neg\neg(K(A) \wedge K(B))}
						}{K(A), K(B) \Rightarrow K(A \wedge B)}
					}{K(A) \wedge K(B) \Rightarrow K(A \wedge B)}
				}{K(\Gamma) \Rightarrow K(A \wedge B)}
				\)}\]
			
			\noindent
			\textbf{The rules \textnormal{R$\top_{(K,K)}$}, \textnormal{R$\vee_{1(K,K)}$},
				\textnormal{R$\vee_{2(K,K)}$}, and \textnormal{R$\exists_{(K,K)}$}.} Can be handled analogously.
			
			\noindent
			\textbf{The rule \textnormal{L$\wedge_{(K,K)}$}.}
			Next, consider the rule
			\[
			\prftree[r]{L$\wedge_{(K,K)}$}{K(A),K(B),K(\Gamma)\Rightarrow K(C)}{K(A \wedge B),K(\Gamma)\Rightarrow K(C)}
			\]
			We obtain:
			\[
			\prftree[r]{dfn}{
				\prftree[r]{Lp$(1,\neg\neg)$}{
					\prftree[r]{L$\wedge$}{K(A), K(B),K(\Gamma)\Rightarrow K(C)}{K(A) \wedge K(B),K(\Gamma)\Rightarrow K(C)}
				}{\neg\neg (K(A) \wedge K(B)),K(\Gamma)\Rightarrow K(C)}
			}{K(A \wedge B),K(\Gamma)\Rightarrow K(C)}
			\]
			where we can apply Lp$(1,\neg\neg)$ because
			$K(C)=\neg\neg C'$ for some formula $C'$.
			
			\noindent
			\textbf{The rules 
				\textnormal{L$\vee_{(K,K)}$},
				\textnormal{L$\to_{(K,K)}$},
				\textnormal{L$\exists_{(K,K)}$},
				\textnormal{L$\forall_{(K,K)}$}, and
				\textnormal{LE$_{(K,K)}$}.} Can be handled analogously.
			
			\noindent
			\textbf{The rule \textnormal{R$\to_{(K,K)}$}.}
			Let us consider the rule
			\[
			\prftree[r]{R$\to_{(K,K)}$}{K(A), K(\Gamma) \Rightarrow K(B)}{K(\Gamma) \Rightarrow K(A\to B)}
			\]
			We obtain:
			\[
			\prftree[r]{dfn}{
				\prftree[r]{R$(1,\neg\neg)$}{
					\prftree[r]{R$\to$}{
						{K(A), K(\Gamma) \Rightarrow K(B)}
					}{K(\Gamma) \Rightarrow K(A)\to K(B)}
				}{K(\Gamma) \Rightarrow \neg\neg (K(A)\to K(B))}
			}{K(\Gamma) \Rightarrow K(A\to B)}
			\]
			
			\noindent
			\textbf{The rule \textnormal{R$\forall_{(K,K)}$}.} Can be handled in the same manner.
		\end{proof}
		
		We can now retrieve {Kolmogorov's} negative translation theorem.
		
		\begin{prop}[Kolmogorov \cite{Kolmogorov25}]\
			\label{prop-kolmogorov}
			The function $K$ is a sound and faithful translation from $\Gc$ to $\Gm$.
		\end{prop}
		
		\begin{proof}
			Direct application of Theorem~\ref{thm-kolmogorov} given Lemmata~\ref{rmk-Kolmogorov-Gentzen}~and~\ref{lem-Kolmogorov-Gentzen}.
		\end{proof}
		
		Analogously, our result can also be used to derive other negative translation theorems, 
		such as those established by Gödel~\cite{Godel33}, Gentzen~\cite{Gentzen36} and Kuroda~\cite{Kuroda51}.
		
		\section{Kuroda's translation theorem with minimality}
		\label{sec-Kuroda-min}
		
		Notice that many negative translation theorems were originally formulated and proved over intuitionistic logic rather than minimal logic. Some translations—such as Gödel's and Gentzen's—can indeed be pushed down to the minimal setting, but this is not true in general. For instance, Kuroda's does not work on minimal logic.
		Therefore, we want to apply our results to describe the least logic for which Kuroda's double negation theorem holds.
		
		Consider \Gm, and define
		\[
		\begin{array}{rclcrcl}
			J(P) &:=& P, 
			\\[1mm]
			J(A * B) &:=& J(A) * J(B) &\text{if } * \in \{\wedge, \vee, \to\}, 
			\\[1mm]
			J(\exists x\, A) &:=& \exists x\, J(A), 
			\\[1mm]
			J(\forall x\, A) &:=& \forall x\, \neg\neg J(A),
			\\[1mm]
			J(\epsilon) &:=& \epsilon,
			\\[1mm]
			J(C_1,\ldots,C_n) &:=& J(C_1), \ldots, J(C_n).
		\end{array}
		\]
		We observe that, for every formula $A$, and every term $t$,
		\(
		J(A[t/x])=(J(A))[t/x].
		\)
		In particular, the eigenvariable conditions of the quantifier rules are
		preserved under the transformation $J$.
		
		Let us define the calculus \Gk\ as \(\Gm\), plus Cut and
		\[
		\prftree[r]{R$\to_{(J,\neg\neg J)}$}
		{J(A),J(\Gamma)\Rightarrow\neg\neg J(B)}
		{J(\Gamma)\Rightarrow\neg\neg J(A\to B)}.
		\]
		We will dub this \emph{Kuroda's logic}.
		Let us state and prove the following version of Kuroda's negative translation theorem:
		
		\begin{prop}\footnote{
				In~\cite{Fellin26}, the corresponding application was formulated
				using $\neg\neg J$ on both sides, due to the limitations of the framework
				developed there. Here we instead use $(J,\neg\neg J)$. The two formulations
				are equivalent, so this change does not affect the resulting translation
				theorem. We adopt the latter formulation here because it provides a
				genuine application of the more general sequent-level framework developed
				in the present paper.
			}
			\label{prop-Kuroda}
			Let $\K$ be a calculus such that
			$\Gm\preceq\K\preceq\Gcs$.
			Then:
			\begin{enumerate}
				\item If $\K\vdash J(\Gamma) \Rightarrow \neg\neg J(A)$, then $\Gcs\vdash \Gamma \Rightarrow A$.
				\item The following are equivalent:
				\begin{enumerate}
					\item If $\Gcs\vdash \Gamma \Rightarrow A$, then $\K\vdash J(\Gamma) \Rightarrow \neg\neg J(A)$.
					\item $\Gk\preceq\K$.
				\end{enumerate}
			\end{enumerate}
		\end{prop}
		
		\begin{proof}
			The claim follows by verifying all the conditions of
			Theorem~\ref{thm-main}.
			
			As in Lemma~\ref{lem-Kolmogorov-Gentzen}, an easy induction shows that
			$J$ is stable and expansive in $\Gcs$, using the fact that $\neg\neg$
			can be ``freely added and removed.'' Consequently, the same holds for
			$\neg\neg J$. Hence, by Remark~\ref{rmk-jstable}, the map
			$(J,\neg\neg J)\colon\Gcs\to\Gcs$ is sound and faithful.
			
			It remains to show that
			$\Gcs_{(J,\neg\neg J)}\preceq\Gk.$
			
			\noindent
			\textbf{The rule \textnormal{R$\wedge_{(J,\neg\neg J)}$}.}
			Let us first consider the rule
			\[
			\prftree[r]{R$\wedge_{(J,\neg\neg J)}$}
			{J(\Gamma) \Rightarrow \neg\neg J(A)}{}
			{J(\Gamma) \Rightarrow \neg\neg J(B)}
			{J(\Gamma) \Rightarrow \neg\neg J(A \wedge B)}
			\]
			We get:
			\[\resizebox{\linewidth}{!}{\(
				\prftree[r]{\Cut}{
					\prftree[r]{R$\wedge$}{
						\prfassumption{J(\Gamma) \Rightarrow \neg\neg J(A)}
					}{
						\prfassumption{J(\Gamma) \Rightarrow \neg\neg J(B)}
					}{J(\Gamma) \Rightarrow \neg\neg J(A) \wedge \neg\neg J(B)}
				}{
					\prftree[r]{L$\wedge$}{
						\prftree[r]{dfn}{
							\prftree[r]{Lp$(1,\neg\neg)$}{
								\prftree[r]{R$(1,\neg\neg)$}{
									\prftree[r]{R$\wedge$}{
										\prftree[r]{\textnormal{Refl$^+$}}{J(A), J(B) \Rightarrow J(A)}
									}{
										\prftree[r]{\textnormal{Refl$^+$}}{J(A), J(B) \Rightarrow J(B)}
									}{J(A), J(B) \Rightarrow J(A) \wedge J(B)}
								}{J(A), J(B) \Rightarrow \neg\neg(J(A) \wedge J(B))}
							}{\neg\neg J(A), \neg\neg J(B) \Rightarrow \neg\neg(J(A) \wedge J(B))}
						}{\neg\neg J(A), \neg\neg J(B) \Rightarrow \neg\neg J(A \wedge B)}
					}{\neg\neg J(A) \wedge \neg\neg J(B) \Rightarrow \neg\neg J(A \wedge B)}
				}{J(\Gamma) \Rightarrow \neg\neg J(A \wedge B)}
				\)}\]
			
			\noindent
			\textbf{The rules \textnormal{R$\top_{(J,\neg\neg J)}$},
				\textnormal{R$\vee_{1(J,\neg\neg J)}$},
				\textnormal{R$\vee_{2(J,\neg\neg J)}$}, and
				\textnormal{R$\exists_{(J,\neg\neg J)}$}.} 
			Can be handled analogously.
			
			\noindent
			\textbf{The rules 
				\textnormal{L$\wedge_{(J,\neg\neg J)}$},
				\textnormal{L$\vee_{(J,\neg\neg J)}$},
				\textnormal{L$\exists_{(J,\neg\neg J)}$},
				\textnormal{L$\forall_{(J,\neg\neg J)}$}, and
				\textnormal{LE$_{(J,\neg\neg J)}$}.}
			These are instances of \textnormal{L$\wedge_{(J,\neg\neg J)}$},
			\textnormal{L$\vee_{(J,\neg\neg J)}$},
			\textnormal{L$\exists_{(J,\neg\neg J)}$},
			\textnormal{L$\forall_{(J,\neg\neg J)}$}, and
			\textnormal{LE$_{(J,\neg\neg J)}$}, respectively.
			
			\noindent
			\textbf{The rule \textnormal{R$\to_{(J,\neg\neg J)}$}.}
			Let us now consider the rule
			\[
			\prftree[r]{R$\to_{(J,\neg\neg J)}$}
			{J(\Gamma) \Rightarrow \neg\neg J(A)}
			{J(B), J(\Gamma) \Rightarrow \neg\neg J(C)}
			{J(A\to B), J(\Gamma) \Rightarrow \neg\neg J(C)}
			\]
			We get:
			\[
			\resizebox{\linewidth}{!}{\(
				\prftree[r]{dfn}{
					\prftree[r]{\Cut}{
						\prftree[r]{R$\to$}{
							\prftree[r]{Lp$(1,\neg\neg)$}{
								\prftree[r]{R$(1,\neg\neg)$}{
									\prftree[r]{LE}{
										\prftree[r]{L$\to$}{
											\prftree[r]{Refl$^+$}{J(A) \Rightarrow J(A)}
										}{
											\prftree[r]{Refl$^+$}{J(B), J(A) \Rightarrow J(B)}
										}{J(A)\to J(B), J(A) \Rightarrow J(B)}
									}{J(A), J(A)\to J(B) \Rightarrow J(B)}
								}{J(A), J(A)\to J(B) \Rightarrow \neg\neg J(B)}
							}{\neg\neg J(A), J(A)\to J(B) \Rightarrow \neg\neg J(B)}
						}{J(A)\to J(B) \Rightarrow \neg\neg J(A)\to \neg\neg J(B)}
					}{
						\prftree[r]{L$\to$}{J(\Gamma)\Rightarrow \neg\neg J(A)}{
							\prftree[r]{Lp$(1,\neg\neg)$}{
								J(B), J(\Gamma) \Rightarrow \neg\neg J(C)
							}{\neg\neg J(B), J(\Gamma) \Rightarrow \neg\neg J(C)}
						}{\neg\neg J(A)\to \neg\neg J(B), J(\Gamma) \Rightarrow \neg\neg J(C)}
					}{J(A)\to J(B), J(\Gamma) \Rightarrow \neg\neg J(C)}
				}{J(A\to B), J(\Gamma) \Rightarrow \neg\neg J(C)}
				\)}
			\]

			\noindent
			\textbf{The rule \textnormal{R$\forall_{(J,\neg\neg J)}$}.}
			Let us now consider the rule
			\[
			\prftree[r]{R$\forall_{(J,\neg\neg J)}$}
			{J(\Gamma) \Rightarrow \neg\neg J(A[y/x])}
			{J(\Gamma) \Rightarrow \neg\neg J(\forall x\, A)} 
			\quad(y\ \text{fresh})
			\]
			We get:
			\[
			\prftree[r]{R$(1,\neg\neg)$}{
				\prftree[r]{dfn}{
					\prftree[r]{R$\forall$}
					{J(\Gamma) \Rightarrow \neg\neg J(A[y/x])}
					{J(\Gamma) \Rightarrow \forall x\, \neg\neg J(A)}
				}{J(\Gamma) \Rightarrow J(\forall x\, A)}
			}{J(\Gamma) \Rightarrow \neg\neg J(\forall x\, A)}
			\]
			
			\noindent
			\textbf{The rule \textnormal{R$\to_{(J,\neg\neg J)}$}.} Admissible by definition of
			$\Gk$. 
			
			\noindent
			\textbf{The rule \textnormal{L$\neg\neg_{(J,\neg\neg J)}$}.} Admissible as an
			instance of Lp$(1,\neg\neg)$, since $\neg\neg$ is a nucleus in $\Gk$.
			
			\noindent
			\textbf{The rule \textnormal{IS$_{(J,\neg\neg J)}$}.} Follows directly from
			R$(1,\neg\neg)$.
			
			\noindent
			\textbf{The rule \Cut$_{(J,\neg\neg J)}$.}
			Finally, consider the rule
			\[
			\prftree[r]{\Cut$_{(J,\neg\neg J)}$}
			{J(\Gamma)\Rightarrow \neg\neg J(A)}
			{J(\Delta_1),J(A),J(\Delta_2)\Rightarrow \neg\neg J(B)}
			{J(\Delta_1),J(\Gamma),J(\Delta_2)\Rightarrow \neg\neg J(B)}
			\]
			We get:
			\[
			\prftree[r]{\Cut}
			{J(\Gamma)\Rightarrow \neg\neg J(A)}
			{
				\prftree[r]{Lp$(1,\neg\neg)$}
				{J(\Delta_1),J(A),J(\Delta_2)\Rightarrow \neg\neg J(B)}
				{J(\Delta_1),\neg\neg J(A),J(\Delta_2)\Rightarrow \neg\neg J(B)}
			}
			{J(\Delta_1),J(\Gamma),J(\Delta_2)\Rightarrow \neg\neg J(B)}
			\qedhere
			\]
		\end{proof}
		
		In the final part of this section, we show that~$\Gk$ represents a logic
		distinct from those represented by the calculi introduced so far.
		
		Clearly, we have $\Gm \subseteq \Gk$, but we now show that this inclusion
		is strict. Consider a quantifier-free formula~$A$. By the definition of~$J$,
		we have $J(A)=A$. We then obtain:
		\[
		\prftree[r]{\Cut}{
			\prftree[r]{R$(1,\neg\neg)$}{
				\prftree[r]{R$\to$}{
					\prftree[r]{R$\to$}{
						\prftree[r]{\textnormal{Refl$^+$}}{\neg A, \bot \Rightarrow \bot}
					}{\bot \Rightarrow \neg\neg A}
				}{{}\Rightarrow \bot \to \neg\neg A}
			}{{}\Rightarrow \neg\neg(\bot \to \neg\neg A)}
		}{
			\prftree[r]{R$\to_{(J,\neg\neg J)}$}{
				\prftree[r]{LE}{
					\prftree[r]{Lp$(1,\neg\neg)$}{
						\prftree[r]{L$\to$}{
							\prftree[r]{\textnormal{Refl$^+$}}{\bot \Rightarrow \bot}
						}{
							\prftree[r]{\textnormal{Refl$^+$}}{\neg\neg A , \bot \Rightarrow \neg\neg A}
						}{\bot \to \neg\neg A , \bot \Rightarrow \neg\neg A}
					}{\neg\neg(\bot \to \neg\neg A) , \bot \Rightarrow \neg\neg A}
				}{\bot, \neg\neg(\bot \to \neg\neg A) \Rightarrow \neg\neg A}
			}{\neg\neg(\bot \to \neg\neg A) \Rightarrow \neg\neg(\bot \to A)}
		}{{}\Rightarrow \neg\neg(\bot \to A)}
		\]
		Thus, the double negation of explosion holds for quantifier-free formulae
		in~$\Gk$, whereas this is not the case in~$\Gm$ \cite{Segerberg68,Orevkov68}.
		Hence,
		\[
		\Gm \subsetneq \Gk.
		\]
		
		Next, observe that the rule~R$\to_{(\neg\neg,\neg\neg)}$ is strictly
		stronger than R$\to_{(J,\neg\neg J)}$. Since the former is admissible in
		both $\Gi$ and $\Gg$, it follows that
		\[
		\Gk \subseteq \Gi
		\qquad\text{and}\qquad
		\Gk \subseteq \Gg.
		\]
		As $\Gi$ and $\Gg$ are incomparable \cite{BartoliFellinTesi2026},
		both inclusions are strict. Consequently,
		\[
		\Gm \subsetneq \Gk \subsetneq \Gi,
		\qquad
		\Gk \subsetneq \Gg,
		\]
		and $\Gk$ represents a logic distinct from each of the other logics
		considered so far.
		
		The discussion above is intended as an initial analysis of~$\Gk$.
		We identify $\Gk$ as the minimal calculus validating Kuroda's negative
		translation theorem and show that it does not coincide with any of the
		previously considered systems. A more detailed comparison with these
		and other calculi is left for future investigation.
		
		\section{Conservation classes}
		
		Finally, for Orevkov-style results, we focus on calculi $\L$ and $\M$ on the same domain $S$, and we consider sequents where $\gleft$ does not alter the left-hand side and $\gright$ does not alter the right-hand side. We therefore define classes as follows.
		
		\begin{defi}
			Define $\Class_{(\gleft,\gright)}:=\{\Gamma \Rightarrow \Delta\colon \gleft(\Gamma)=\Gamma\text{ and }\gright(\Delta)=\Delta\}$.
		\end{defi}
		
		\begin{thm}
			\label{thm-orevkov}
			Let $\L$ and $\M$ be calculi on $S$.
			Let $\gleft,\gright$ such that $\M_{(\gleft,\gright)} \preceq \L$.
			Consider a sequent $\Gamma\Rightarrow\Delta$ in $\Class_{(\gleft,\gright)}$. Then
			\[ \text{If } \M\vdash\Gamma\Rightarrow\Delta \text{, then } \L\vdash\Gamma\Rightarrow\Delta. \]
			If, moreover, $\L\leq\M$, then also the converse holds.
		\end{thm}
		
		\begin{proof}
			Suppose that $\M\vdash\Gamma\Rightarrow\Delta$. By Theorem \ref{thm-main}, we get $\L\vdash\gleft(\Gamma)\Rightarrow\gright(\Delta)$. By assumption, the latter is the same as $\L\vdash\Gamma\Rightarrow\Delta$.
			The second statement is immediate.
		\end{proof}
		
		\subsection{Glivenko's second theorem}
		
		We now show how this corollary applies to logic in order to obtain the alternative version of Glivenko's theorem. We begin with a lemma.
		Let $\gleft:=1$, and
		\[
		\begin{array}{rclcrcl}
			\gright(A)&:=&\begin{cases}
				A
				&\text{if }A=\neg B\text{ for some }B,
				\\
				\neg\neg A
				&\text{otherwise,}
			\end{cases}
			\\[1mm]
			\gright(\epsilon) &:=& \bot,
			\\[1mm]
			\gright(C_1,\ldots,C_n)&:=&\gright(C_1\vee\cdots\vee C_n).
		\end{array}
		\]
		
		\begin{lem}\
			\label{lem-glivenko-neg}
			$(\Gcp)_{(\gleft,\gright)} \preceq \Gip$.
		\end{lem}
		
		\begin{proof}
			Once observed that $\neg\neg\neg A$ is minimally equivalent to $\neg A$, this follows from Lemma \ref{lem-glivenko}.
		\end{proof}
		
		\begin{prop}[Glivenko \cite{Glivenko29}]
			$\Gcp \vdash \Gamma \Rightarrow \neg A$ if and only if $\Gip \vdash \Gamma \Rightarrow \neg A$.
		\end{prop}
		\begin{proof}
			Direct application of Theorem~\ref{thm-orevkov} given Lemma~\ref{lem-glivenko-neg}.
		\end{proof}
		
		\subsection{Orevkov's theorems on conservation classes 1--4}
		\label{sec-Orevkov}
		
		To illustrate the applicability of our framework, we now show how the multi-succedent counterparts of Orevkov's theorems on classes~1--4 can be derived as instances of our main result.
		
		The underlying strategy is straightforward. We deal with translations $(\gleft_i,\gright_i)$ ($i\in\{1,\ldots,4\}$), where $\gleft_i$ is designed to act on formulae occurring in negative position, while $\gright_i$ acts on formulae occurring in positive position. Whenever a formula $F$ is not permitted to occur negatively, we define $\gleft_i(F):=\bot$; dually, whenever $F$ is not permitted to occur positively, we define $\gright_i(F):=\top$. In all other cases, $\gleft_i$ and $\gright_i$ are propagated recursively so as to preserve their respective negative and positive character. Consequently, any formula that does not involve the excluded connectives or quantifiers remains unchanged under the translation.%
		\footnote{
			Our approach is closely related in spirit to the method developed by
			Negri~\cite{Negri03,Negri16,Negri21}. Both methods exploit the behaviour
			of critical operators in order to obtain conservativity results.
			However, they deal with critical operators in fundamentally different
			ways. Negri's method avoids an occurrence of a critical operator by
			excluding the corresponding inference rule from the derivation. Our
			method instead eliminates the critical operator by translating it to a
			trivial formula, namely $\top$ or $\bot$ according to its polarity.
			Thus, rather than restricting the derivation by forbidding a rule, we
			modify the translation so that the corresponding rule becomes
			admissible for trivial reasons.
		}
		
		In the present setting, however, the pairs $(\gleft_i,\gright_i)$ are clearly not Lindenbaum-injective. As a result, the faithfulness and minimality statements of Theorem~\ref{thm-main} are not applicable. In particular, $(\gleft_i,\gright_i)$ do not interact well with the \Cut\ rule, making it necessary to work with a cut-free classical calculus, such as the multi-succedent \Gc. Since, however, we are considering single-succedent sequents, we employ the construction described in Remark~\ref{rmk-multisingle}: for a finite succedent $C_1,\ldots,C_n$, we define
		$\gright_i(C_1,\ldots,C_n):=\gright_i(C_1\vee\cdots\vee C_n).$
		
		To sum up, \emph{with the exception of the exclusion clauses discussed above}, the translation $(\gleft_0,\gright_0)$ is then defined as follows.
		\[
		\begin{array}{rclcrclc}
			\gleft_0(P) &:=& P,
			&&
			\gright_0(P) &:=& P,
			\\[1mm]
			\gleft_0(A * B) &:=& \gleft_0(A) * \gleft_0(B),
			&&
			\gright_0(A * B) &:=& \gright_0(A) * \gright_0(B),
			\\[1mm]
			\gleft_0(A \to B) &:=& \gright_0(A) \to \gleft_0(B),
			&&
			\gright_0(A \to B) &:=& \gleft_0(A) \to \gright_0(B),
			\\[1mm]
			\gleft_0(\mathsf{Q} x\, A) &:=& \mathsf{Q} x\, \gleft_0(A),
			&&
			\gright_0(\mathsf{Q} x\, A) &:=& \mathsf{Q} x\, \gright_0(A),
			\\[1mm]
			\gleft_0(\epsilon) &:=& \epsilon,
			&&
			\gright_0(\epsilon) &:=& \bot,
			\\[1mm]
			\gleft_0(C_1,\ldots,C_n) &:=& \gleft_0(C_1),\ldots,\gleft_0(C_n),
			&&
			\gright_0(C_1,\ldots,C_n) &:=& \gright_0(C_1 \vee\cdots\vee C_n).
		\end{array}
		\]
		where $*\in\{\wedge,\vee\}$ and $\mathsf{Q}\in\{\forall,\exists\}$.
		Starting from $(\gleft_0,\gright_0)$, we define
		$(\gleft_i,\gright_i)$, for $i\in\{1,\ldots,4\}$, by replacing the
		clauses specified in the following table:
		\[
		\renewcommand{\arraystretch}{1.5}
		\begin{array}{r|rcl|rcl|}
			i &
			\multicolumn{3}{c|}{\text{modified clauses of }\gleft_i} &
			\multicolumn{3}{c|}{\text{modified clauses of }\gright_i}
			\\
			\hline
			1
			&
			\text{(none)}&&&
			\gright_1(A\to B) &:=& \top,
			\\&&&&
			\gright_1(\forall x\,A) &:=& \top
			\\[1mm]\hline
			2
			&
			\gleft_2(A\vee B)&:=&\bot
			&
			\gright_2(A\to B)&:=&\top
			\\[1mm]\hline
			3
			&
			\gleft_3(\forall x\,A)&:=&\bot
			&
			\gright_3(A\to B)&:=&\top
			\\[1mm]\hline
			4
			&
			\gleft_4(A\to B)&:=&\bot
			&
			\gright_4(A\vee B) &:=& \top,
			\\&&&&
			\gright_4(\exists x\,A) &:=& \top
			\\[1ex]
			\hline
		\end{array}
		\]
		
		We observe that, for every $i\in\{1,\ldots,4\}$, every formula $A$, and every
		term $t$,
		\[
		\gleft_i(A[t/x])=(\gleft_i(A))[t/x]
		\qquad\text{and}\qquad
		\gright_i(A[t/x])=(\gright_i(A))[t/x].
		\]
		In particular, the eigenvariable conditions of the quantifier rules are
		preserved under the transformations $\gleft_i$ and $\gright_i$.
		
		In each case, the goal is to prove that, if the sequent $\Gamma \Rightarrow A$ is in $\Class_{(\gleft_i,\gright_i)}$, then $\Gc \vdash \Gamma \Rightarrow A$ if and only if $\Gi \vdash \Gamma \Rightarrow A$.
		The idea is that, once we know $\Gc_{(\gleft,\gright)}\preceq\Gi$, then such statement follows as a direct application of Theorem~\ref{thm-orevkov}.
		
		\begin{lem}
			\label{lem-Orevkov}
			For every $i\in\{1,\ldots,4\}$, the rules
			\textnormal{IS$_{(\gleft_i,\gright_i)}$}, 
			\textnormal{LE$_{(\gleft_i,\gright_i)}$}, 
			\textnormal{RE$_{(\gleft_i,\gright_i)}$}, 
			\textnormal{L$\bot_{(\gleft_i,\gright_i)}$}, 
			\textnormal{R$\top_{(\gleft_i,\gright_i)}$}, 
			\textnormal{L$\wedge_{(\gleft_i,\gright_i)}$},  
			\textnormal{R$\wedge_{(\gleft_i,\gright_i)}$},
			\textnormal{L$\vee_{(\gleft_i,\gright_i)}$}, 
			\textnormal{R$\vee_{(\gleft_i,\gright_i)}$},
			\textnormal{L$\forall_{(\gleft_i,\gright_i)}$}, and
			\textnormal{L$\exists_{(\gleft_i,\gright_i)}$}
			are admissible in {\Gi}.
		\end{lem}
		\begin{proof}
			\textbf{The rule $\textnormal{IS}^+_{(\gleft_i,\gright_i)}$.}	
			We first establish the admissibility of the single-succedent version
			\[
			\prftree[r]{\textnormal{IS}$^s_{(\gleft_i,\gright_i)}$}{}
			{\gleft_i(\Gamma_1),\gleft_i(P),\gleft_i(\Gamma_2)\Rightarrow\gright_i(P)}
			\]
			of $\textnormal{Refl}^+_{(\gleft_i,\gright_i)}$:
			\[
			\prftree[r]{\textnormal{dfn}}{
				\prftree[r]{\textnormal{Refl$^+$}}{}
				{\gleft_i(\Gamma_1),P,\gleft_i(\Gamma_2)\Rightarrow P}
			}{
				\gleft_i(\Gamma_1),\gleft_i(P),\gleft_i(\Gamma_2)
				\Rightarrow\gright_i(P)
			}
			\]
			
			We now prove the admissibility of the more general rule
			$\textnormal{IS}_{(\gleft_i,\gright_i)}$.
			There are three cases.
			
			\begin{enumerate}
				\item If both $\Delta_1$ and $\Delta_2$ are empty, we can directly
				apply $\textnormal{IS}^s_{(\gleft_i,\gright_i)}$.
				
				\item If at least one of $\Delta_1$ and $\Delta_2$ is inhabited and
				$\gright_i(B\vee C)=\top$, the result follows directly by R$\top$.
				
				\item If at least one of $\Delta_1$ and $\Delta_2$ is inhabited and
				$\gright_i(B\vee C)\neq\top$, we first apply
				$\textnormal{IS}^s_{(\gleft_i,\gright_i)}$ and then apply
				R$\vee_1$ and/or R$\vee_2$ as many times as necessary to introduce
				the contexts $\Delta_1$ and $\Delta_2$.
			\end{enumerate}
			
			\noindent
			\textbf{The rule $\textnormal{LE}_{(\gleft_i,\gright_i)}$.}	
			The rule $\textnormal{LE}_{(\gleft_i,\gright_i)}$ is an instance of LE.
			
			\noindent
			\textbf{The rule $\textnormal{RE}_{(\gleft_i,\gright_i)}$.}	
			If $\gright_i(A\vee B)=\top$, then the conclusion of
			$\textnormal{RE}_{(\gleft_i,\gright_i)}$ is directly derivable by
			R$\top$. We may therefore assume that
			$\gright_i(A\vee B)\neq\top$.
			
			First consider the case in which both $\Delta_1$ and $\Delta_2$ are
			empty. We have
			\begin{align*}
				\prftree[r]{\textnormal{Cut}}{
					\gleft_i(\Gamma)\Rightarrow
					\gright_i(B)\vee\gright_i(A)
				}{
					\prftree[r]{}{
						\gright_i(B)\vee\gright_i(A)
						\Rightarrow
						\gright_i(A)\vee\gright_i(B)
					}
				}{
					\gleft_i(\Gamma)\Rightarrow
					\gright_i(A)\vee\gright_i(B)
				}
				\tag{$\star$}
			\end{align*}
			where the right premiss is a standard derivable sequent.
			
			The cases in which one or both of the contexts are inhabited are
			handled analogously, using the following derivable sequents:
			\begin{align*}
				\Gi &\vdash
				D_1\vee A\vee B\vee D_2
				\Rightarrow
				D_1\vee B\vee A\vee D_2,
				&&\text{if both $\Delta_1$ and $\Delta_2$ are inhabited},\\
				\Gi &\vdash
				D_1\vee A\vee B
				\Rightarrow
				D_1\vee B\vee A,
				&&\text{if $\Delta_1$ is inhabited and $\Delta_2$ is not},\\
				\Gi &\vdash
				A\vee B\vee D_2
				\Rightarrow
				B\vee A\vee D_2,
				&&\text{if $\Delta_2$ is inhabited and $\Delta_1$ is not}.
			\end{align*}
			
			\noindent
			\textbf{The rule $\textnormal{L$\bot$}_{(\gleft_i,\gright_i)}$.} This is an instance of
			L$\bot$.
			
			\noindent
			\textbf{The rule $\textnormal{R$\top$}_{(\gleft_i,\gright_i)}$.}	
			First suppose that $\gright_i(A\vee B)\neq\top$.
			If $\Delta$ is empty, the rule is an instance of R$\top$.
			If $\Delta$ is inhabited, we have
			\[
			\prftree[r]{\textnormal{dfn}}{
				\prftree[r]{\textnormal{R$\vee_2$}}{
					\prftree[r]{\textnormal{R$\top$}}{
						\gleft_i(\Gamma)\Rightarrow\top
				}}{
					\gleft_i(\Gamma)\Rightarrow
					\gright_i(\Delta)\vee\top
				}
			}{
				\gleft_i(\Gamma)\Rightarrow
				\gright_i(\Delta,\top)
			}
			\]
			If $\gright_i(A\vee B)=\top$, the rule is always an instance of R$\top$.
			
			\noindent
			\textbf{The rules $\textnormal{L*}_{(\gleft_i,\gright_i)}$,
				for $*\in\{\vee,\wedge\}$.}	
			If $\gleft_i(A*B)\neq\bot$, the rule is an instance of L$*$.
			Otherwise, its conclusion is derivable from L$\bot$.
			
			\noindent
			\textbf{The rule $\textnormal{R$\wedge$}_{(\gleft_i,\gright_i)}$.}	
			We distinguish three cases.
			
			\begin{itemize}
				\item Suppose that
				$\gright_i(A\vee B)\neq\top$ and
				$\gright_i(A\wedge B)\neq\top$.
				If $\Delta$ is inhabited, admissibility follows by the same pattern
				as in $(\star)$, using the derivable sequent
				\[
				\Gi\vdash
				(D\vee A)\wedge(D\vee B)
				\Rightarrow
				D\vee(A\wedge B).
				\]
				If $\Delta$ is empty, then
				$\textnormal{R$\wedge$}_{(\gleft_i,\gright_i)}$ is an instance of
				R$\wedge$.
				
				\item Suppose that
				$\gright_i(A\vee B)\neq\top$ and
				$\gright_i(A\wedge B)=\top$.
				This is handled exactly as the case
				$\gright_i(A\vee B)\neq\top$ of
				$\textnormal{R$\top$}_{(\gleft_i,\gright_i)}$.
				
				\item Suppose that $\gright_i(A\vee B)=\top$.
				If $\Delta$ is inhabited, the conclusion is directly derivable by
				R$\top$. Otherwise,
				$\textnormal{R$\wedge$}_{(\gleft_i,\gright_i)}$ is an instance of
				R$\wedge$.
			\end{itemize}
			
			\noindent
			\textbf{The rule $\textnormal{R$\vee$}_{(\gleft_i,\gright_i)}$.}	
			If $\gright_i(A\vee B)\neq\top$, the premiss and conclusion are
			definitionally the same. Otherwise, the conclusion is directly
			derivable by R$\top$.
			
			\noindent
			\textbf{The rules
				$\textnormal{L$\mathsf{Q}$s}_{(\gleft_i,\gright_i)}$,
				for $\mathsf{Q}\in\{\forall,\exists\}$.}	
			If $\gleft_i(\mathsf{Q}x\,A)\neq\bot$, the rule is an instance of
			L$\mathsf{Q}$. Otherwise, its conclusion is derivable from L$\bot$.
			\qedhere
		\end{proof}
		
		What about R$\to_{(\gleft_i,\gright_i)}$, R$\forall_{(\gleft_i,\gright_i)}$,
		R$\exists_{(\gleft_i,\gright_i)}$, and L$\to_{(\gleft_i,\gright_i)}$?
		In general, their admissibility cannot be established solely from
		the preceding argument.
		
		\noindent
		\textbf{The rule $\textnormal{R$\forall$}_{(\gleft_i,\gright_i)}$.}
		We distinguish three cases.
		\begin{itemize}
			\item Suppose that
			$\gright_i(A\vee B)\neq\top$ and
			$\gright_i(\forall x\,A)=\top$.
			This is handled exactly as the case
			$\gright_i(A\vee B)\neq\top$ of
			$\textnormal{R$\top$}_{(\gleft_i,\gright_i)}$.
			
			\item Suppose that $\gright_i(A\vee B)=\top$.
			If $\Delta$ is inhabited, the conclusion of
			$\textnormal{R$\forall$}_{(\gleft_i,\gright_i)}$ is directly derivable
			by R$\top$. If $\Delta$ is empty, then
			$\textnormal{R$\forall$}_{(\gleft_i,\gright_i)}$ is an instance of
			R$\forall$.
			
			\item Suppose that
			$\gright_i(A\vee B)\neq\top$ and
			$\gright_i(\forall x\,A)\neq\top$.
			If $\Delta$ is empty, then
			$\textnormal{R$\forall$}_{(\gleft_i,\gright_i)}$ is an instance of
			R$\forall$.
			If $\Delta$ is inhabited, however, admissibility cannot be established
			in general from the preceding argument, as it depends on the particular
			clauses defining $(\gleft_i,\gright_i)$.
		\end{itemize}
		
		\noindent
		\textbf{The rule $\textnormal{R$\to$}_{(\gleft_i,\gright_i)}$.}
		Again, we distinguish three cases.
		\begin{itemize}
			\item Suppose that
			$\gright_i(A\vee B)\neq\top$ and
			$\gright_i(A\to B)=\top$.
			This is handled exactly as the case
			$\gright_i(A\vee B)\neq\top$ of
			$\textnormal{R$\top$}_{(\gleft_i,\gright_i)}$.
			
			\item Suppose that $\gright_i(A\vee B)=\top$.
			If $\Delta$ is inhabited, the conclusion of
			$\textnormal{R$\to$}_{(\gleft_i,\gright_i)}$ is directly derivable
			by R$\top$. If $\Delta$ is empty, then
			$\textnormal{R$\to$}_{(\gleft_i,\gright_i)}$ is an instance of
			R$\to$.
			
			\item Suppose that
			$\gright_i(A\vee B)\neq\top$ and
			$\gright_i(A\to B)\neq\top$.
			If $\Delta$ is empty, then
			$\textnormal{R$\to$}_{(\gleft_i,\gright_i)}$ is an instance of
			R$\to$.
			If $\Delta$ is inhabited, however, admissibility cannot be established
			in general from the preceding argument, as it depends on the particular
			clauses defining $(\gleft_i,\gright_i)$.
		\end{itemize}
		
		\noindent
		\textbf{The rule $\textnormal{R$\exists$}_{(\gleft_i,\gright_i)}$.}	
		We distinguish three cases.	
		\begin{itemize}
			\item Suppose that
			$\gright_i(A\vee B)\neq\top$ and
			$\gright_i(\exists x\,A)\neq\top$.
			Admissibility follows by an argument analogous to $(\star)$.
			
			\item Suppose that
			$\gright_i(A\vee B)\neq\top$ and
			$\gright_i(\exists x\,A)=\top$.
			This is handled exactly as the case
			$\gright_i(A\vee B)\neq\top$ of
			$\textnormal{R$\top$}_{(\gleft_i,\gright_i)}$.
			
			\item Suppose that $\gright_i(A\vee B)=\top$.
			If $\Delta$ is inhabited, the conclusion is directly derivable by
			R$\top$. 
			If $\Delta$ is empty, however, the premiss might have been derived by R$\top$, and admissibility cannot be established from the preceding argument.
		\end{itemize}
		
		\noindent
		\textbf{The rule $\textnormal{L$\to$}_{(\gleft_i,\gright_i)}$.}	
		We distinguish three cases.	
		\begin{itemize}
			\item Suppose that
			$\gright_i(A\vee B)\neq\top$ and
			$\gleft_i(A\to B)\neq\bot$.
			If $\Delta$ is inhabited, admissibility follows from admissibility of
			\[
			\prftree[r]{}{\Gamma\Rightarrow C\vee A}{B,\Gamma\Rightarrow C}{A\to B,\Gamma\Rightarrow C}
			\]
			which is an easy consequence of $\Gi\vdash C\vee A,B\to C,A\to B\Rightarrow C$ and Cut.
			If $\Delta$ is empty, $\textnormal{L$\to$}_{(\gleft_i,\gright_i)}$ is an instance of L$\to$. 
			
			\item Suppose that
			$\gright_i(A\vee B)\neq\top$ and
			$\gleft_i(A\to B)=\bot$.
			Then the conclusion is derivable by L$\bot$.
			
			\item Suppose that $\gright_i(A\vee B)=\top$.
			If $\Delta$ has at least two elements, the conclusion is directly derivable by
			R$\top$.
			If $\Delta$ is empty, one can just apply L$\to$.
			If $\Delta$ has exactly one element, however, the left premiss might have been derived by R$\top$, and admissibility cannot be established from the preceding argument.
		\end{itemize}
		
		Therefore, combining Lemma~\ref{lem-Orevkov} with the observations above,
		we obtain the following useful reduction.
		
		\begin{rem}
			\label{rmk-orevkov}
			To prove
			$\Gc_{(\gleft_i,\gright_i)}\preceq\Gi$,
			it suffices to verify, in $\Gi$, the admissibility of the four critical
			rules listed below. The conditions in the second column specify the
			instances of the rules that need to be considered, while those in the
			third column specify the translations $(\gleft_i,\gright_i)$ for which
			the corresponding rule may fail to be admissible:
			
			
			\centering
			\begin{tabular}{c|c|c}
				\textbf{Rule}
				& \textbf{instances to be considered}
				& \textbf{potentially critical clauses}
				\\ \hline
				R$\to_{(\gleft_i,\gright_i)}$
				& $\Delta$ inhabited
				& $\gright_i(A\vee B)\neq\top$ and
				$\gright_i(A\to B)\neq\top$
				\\[2mm]
				R$\forall_{(\gleft_i,\gright_i)}$
				& $\Delta$ inhabited
				& $\gright_i(A\vee B)\neq\top$ and
				$\gright_i(\forall x\,A)\neq\top$
				\\[2mm]
				R$\exists_{(\gleft_i,\gright_i)}$
				& $\Delta$ empty
				& $\gright_i(A\vee B)=\top$ and
				$\gright_i(\exists x\,A)\neq\top$
				\\[2mm]
				L$\to_{(\gleft_i,\gright_i)}$
				& $\Delta$ has exactly one element
				& $\gright_i(A\vee B)=\top$ and
				$\gleft_i(A\to B)\neq\bot$
			\end{tabular}
		\end{rem}
		
		\begin{prop}[Orevkov \cite{Orevkov68}, Class 1]\
			\label{prop-Orevkov}
			Suppose that the sequent $\Gamma \Rightarrow A$ is in $\Class_{(\gleft_1,\gright_1)}$.
			Then $\Gc \vdash \Gamma \Rightarrow A$ if and only if $\Gi \vdash \Gamma \Rightarrow A$.
		\end{prop}
		
		\begin{proof}
			Direct consequence of Theorem~\ref{thm-orevkov} and Remark \ref{rmk-orevkov}.
		\end{proof}
		
		\begin{lem}
			\label{lem-Harrop}
			If $\Gi\vdash \gleft_2(\Gamma) \Rightarrow \gright_2(D_1,\ldots,D_n)$, then $\Gi\vdash \gleft_2(\Gamma) \Rightarrow \gright_2(D_i)$ for some $i$.
		\end{lem}
		
		\begin{proof}
			We proceed by induction on the derivation of
			\(
			\gleft_2(\Gamma) \Rightarrow \gright_2(D_1,\ldots,D_n).
			\)
			Observe first that, by the definitions of $\gleft_2$ and $\gright_2$, this sequent cannot be obtained by an application of either L$\vee$ or R$\to$.
			In the cases IS, R$\top$, R$\wedge$, R$\vee_1$, R$\vee_2$, R$\exists$, and R$\forall$, we have $n=1$, and hence the desired implication is immediate.
			In the remaining cases, namely LE, L$\bot$, L$\wedge$, L$\to$, L$\exists$, and L$\forall$, the conclusion follows straightforwardly from the induction hypothesis.
		\end{proof}
		
		\begin{prop}[Orevkov \cite{Orevkov68}, Class 2]\
			Suppose that the sequent $\Gamma \Rightarrow A$ is in $\Class_{(\gleft_2,\gright_2)}$.
			Then $\Gc \vdash \Gamma \Rightarrow A$ if and only if $\Gi \vdash \Gamma \Rightarrow A$.
		\end{prop}
		\begin{proof}
			By Theorem~\ref{thm-orevkov} and Remark~\ref{rmk-orevkov}, it suffices
			to show the admissibility of
			R$\forall_{(\gleft_2,\gright_2)}$ in the case where $\Delta=D_1,\ldots,D_n$ is
			inhabited.
			Suppose that
			$\Gi\vdash
			\gleft_2(\Gamma)
			\Rightarrow
			\gright_2(\Delta,A[y/x])$, where $y$ does not appear freely in $\Gamma$ nor $\Delta$.
			By Lemma~\ref{lem-Harrop}, there are two cases.
			
			\begin{itemize}
				\item \textbf{Case
					$\Gi\vdash\gleft_2(\Gamma)\Rightarrow\gright_2(D_i)$ for some
					$1\leq i\leq n$.}
				We obtain
				\[
				\prftree[r]{\textnormal{dfn}}{
					\prftree[r]{\textnormal{R$\vee_1$, R$\vee_2$}}{
						\prftree{}{
							\gleft_2(\Gamma)
							\Rightarrow
							\gright_2(D_i)
						}
					}{
						\gleft_2(\Gamma)
						\Rightarrow
						\gright_2(D_1)\vee\cdots\vee
						\gright_2(D_n)\vee
						\forall x\,\gright_2(A)
					}
				}{
					\gleft_2(\Gamma)
					\Rightarrow
					\gright_2(\Delta,\forall x\,A)
				}
				\]
				
				\item \textbf{Case
					$\Gi\vdash\gleft_2(\Gamma)\Rightarrow\gright_2(A[y/x])$.}
				We obtain
				\[
				\prftree[r]{\textnormal{dfn}}{
					\prftree[r]{\textnormal{R$\vee_2$}}{
						\prftree[r]{\textnormal{R$\forall$}}{
							\prftree{}{
								\gleft_2(\Gamma)
								\Rightarrow
								\gright_2(A[y/x])
							}
						}{
							\gleft_2(\Gamma)
							\Rightarrow
							\forall x\,\gright_2(A)
						}
					}{
						\gleft_2(\Gamma)
						\Rightarrow
						\gright_2(D_1)\vee\cdots\vee
						\gright_2(D_n)\vee
						\forall x\,\gright_2(A)
					}
				}{
					\gleft_2(\Gamma)
					\Rightarrow
					\gright_2(\Delta,\forall x\,A)
				}
				\]
			\end{itemize}
			
			Thus, R$\forall_{(\gleft_2,\gright_2)}$ is admissible whenever
			$\Delta$ is inhabited.
		\end{proof}
		
		\begin{prop}[Orevkov \cite{Orevkov68}, Class 3]\
			Suppose that the sequent $\Gamma \Rightarrow A$ is in $\Class_{(\gleft_3,\gright_3)}$.
			Then $\Gc \vdash \Gamma \Rightarrow A$ if and only if $\Gi \vdash \Gamma \Rightarrow A$.
		\end{prop}
		\begin{proof}
			By Theorem~\ref{thm-orevkov} and Remark~\ref{rmk-orevkov}, it suffices
			to show the admissibility of
			R$\forall_{(\gleft_3,\gright_3)}$ in the case where $\Delta$ is
			inhabited.
			Suppose that
			\[
			\Gi\vdash
			\gleft_3(\Gamma)
			\Rightarrow
			\gright_3(\Delta,A[y/x]),
			\qquad\text{i.e.,}\qquad
			\Gi\vdash
			\gleft_3(\Gamma)
			\Rightarrow
			\gright_3(\Delta)\vee\gright_3(A[y/x]),
			\]
			where $y$ does not occur freely in $\Gamma$ or $\Delta$.
			We prove, by induction on the derivation of this sequent, that
			\(
			\Gi\vdash
			\gleft_3(\Gamma)
			\Rightarrow
			\gright_3(\Delta,\forall x\,A).
			\)
			We distinguish cases according to the last rule in the derivation.
			\begin{itemize}
%
				\item \textbf{\textnormal{IS} and right-introduction rules other than \textnormal{R$\vee_1$} or \textnormal{R$\vee_2$}.}
				None of these cases can occur, since $\gright_3(A[y/x])$ is not
				principal in their conclusions.
				
				\item \textbf{Case \textnormal{L$\forall$}.}
				The conclusion is derivable directly by L$\bot$.
				
				\item \textbf{Case \textnormal{R$\vee_1$}.}
				Write $A=B\vee C$. We first establish the auxiliary sequent
				$\Gi\vdash
				\gright_3(\Delta)\vee\forall x\,\gright_3(B)
				\Rightarrow
				\gright_3(\Delta)\vee\forall x\bigl(\gright_3(B)\vee\gright_3(C)\bigr).$
				Indeed, we have the more general
				\[
				\prftree[r]{\textnormal{L$\vee$}}{
					\prftree[r]{\textnormal{R$\vee_1$}}{
						\prftree[r]{Refl$^+$}{}{
							D\Rightarrow D
						}
					}{
						D\Rightarrow
						D\vee\forall x(B\vee C)
					}
				}{
					\prftree[r]{\textnormal{R$\vee_2$}}{
						\prftree[r]{\textnormal{R$\forall$}}{
							\prftree[r]{\textnormal{R$\vee_1$}}{
								\prftree[r]{\textnormal{L$\forall$}}{
									\prftree[r]{\textnormal{Refl$^+$}}{}{
										B[y/x] \Rightarrow B[y/x]
									}
								}{
									\forall x\,B \Rightarrow B[y/x]
								}
							}{
								\forall x\,B \Rightarrow B[y/x]\vee C[y/x]
							}
						}{
							\forall x\,B \Rightarrow \forall x(B\vee C)
						}
					}{
						\forall x\,B
						\Rightarrow
						D\vee\forall x(B\vee C)
					}
				}{
					D\vee\forall x\,B
					\Rightarrow
					D\vee\forall x(B\vee C)
				}
				\]
				By the induction hypothesis and Cut,
				we obtain
				$\Gi\vdash
				\gleft_3(\Gamma)
				\Rightarrow
				\gright_3(\Delta,\forall x\,(B\vee C)).$
				
				\item \textbf{Case \textnormal{R$\vee_2$}.}
				The argument is analogous to the preceding case.
				
				\item \textbf{Remaining cases.}
				In all remaining cases, $\gright_3(A)[y/x]$ is not principal in
				the last inference. Moreover, $y$ occurs in the premiss only in
				$A[y/x]$. Hence the desired conclusion follows directly from the
				induction hypothesis.
				\qedhere
			\end{itemize}
		\end{proof}
		
		\begin{prop}[Orevkov \cite{Orevkov68}, Class 4]\
			Suppose that the sequent $\Gamma \Rightarrow A$ is in $\Class_{(\gleft_4,\gright_4)}$.
			Then $\Gc \vdash \Gamma \Rightarrow A$ if and only if $\Gi \vdash \Gamma \Rightarrow A$.
		\end{prop}
		
		\begin{proof}
			Direct consequence of Theorem~\ref{thm-orevkov} and Remark \ref{rmk-orevkov}.
		\end{proof}
		
		\section{Final remarks}
		
		A central feature of the framework developed here is that it separates the
		general proof-theoretic mechanism underlying conservation and translation
		results from the particular properties of the logics to which it is
		applied. This makes it possible to treat a variety of classical
		translations and conservation results within a common setting, while also
		making explicit which properties of the underlying calculi are needed in
		each application.
		
		In particular, the present framework
		provides another way of understanding how classical principles can be
		transferred to constructive settings without treating classical and
		constructive logic as entirely separate systems. In this respect, it
		resonates with Fred Richman's conception of intuitionistic logic as a
		generalisation of classical logic~\cite{Richman90,Richman94}.
		
		
		Several directions for future research arise naturally. Since all
		results in this work have been established using purely constructive
		reasoning, the framework is well suited for formalisation in proof
		assistants such as Agda, following the approach of
		\cite{BorsettoFellinUustaluWan2026}. While the various conservation and
		translation theorems considered here follow directly from our main result
		(Theorem~\ref{thm-main}) and its corollaries, their derivation relies on a
		collection of technical lemmata whose formal verification would constitute
		a natural next step.
		
		The applications to Orevkov's theorems also raise a natural question concerning minimal logic. The translations presented here to obtain the intuitionistic versions of Orevkov's results rely essentially on the rule L$\bot$: critical formulae can be translated to $\bot$, after which the corresponding occurrences can be eliminated using L$\bot$. This mechanism is unavailable in minimal logic, where L$\bot$ is not a rule of the calculus. Consequently, the translations presented here do not directly extend to the minimal-logic versions of Orevkov's results. It would be interesting to determine whether analogous translations can be found for minimal logic, possibly by replacing the use of $\bot$ with a different mechanism for eliminating critical operators.
		
		Another natural extension is to consider logical translations between
		distinct languages. Thus far we have focused on translations relating
		classical and constructive logics formulated over the same underlying
		language. Classical examples include the Gödel--McKinsey--Tarski embedding
		of intuitionistic logic into the modal logic $\mathbf{S4}$
		\cite{Godel1933Interpretation,HameenAnttilaVonPlato2023Godel,McKinseyTarski48,Negri2026},
		the standard translation of modal logic into first-order logic
		\cite{Kripke1963,vanBenthem1983}, and Girard's translation of intuitionistic
		logic into linear logic~\cite{Girard1987,FerreiraOlivaProtin26}. Such embeddings
		have long played a central role in proof theory and semantics, yielding
		relative consistency and conservativity results, semantic completeness
		theorems, and transfers of proof-theoretic techniques across different
		logical systems.
		Our framework is sufficiently general to accommodate translations of this
		kind. In particular, soundness follows directly from
		Theorem~\ref{thm-main-t}. Faithfulness, however, is more subtle. The
		criteria developed in this paper are tailored to translations between
		logics over a common language, and we currently lack a general condition
		ensuring faithfulness in the setting of translations between distinct
		logical languages. Identifying such a criterion would allow embeddings
		such as the Gödel--McKinsey--Tarski translation to be treated within the
		present framework.
		
		Although the present work has been motivated by applications in logic, the
		underlying methodology may also prove useful in other areas of mathematics,
		computer science, and formal philosophy. More broadly, it would be
		interesting to develop a proof-relevant refinement of the theory,
		potentially formulated in categorical terms, which could provide a deeper
		understanding of the structures underlying the framework.
		
		\section*{Acknowledgment}
		\noindent The author is grateful to all those who discussed this work at its various stages and generously shared their comments, questions, and suggestions. In particular, the author wishes to thank 
		Thorsten Altenkirch, 
		Ulrich Berger, 
		Hugo Herbelin, 
		Daniel Misselbeck-Wessel, 
		Sara Negri, 
		Peter Schuster, and
		Matteo Tesi.
		This research was partially supported by the \emph{Gruppo Nazionale per le Strutture Algebriche, Geometriche e le loro Applicazioni} (GNSAGA) of the \emph{Istituto Nazionale di Alta Matematica ``Francesco Severi''} (INdAM).
		
		
		
		\bibliographystyle{alpha}
		\bibliography{unifying-BIB}
		
	\end{document}